\documentclass[letterpaper,reqno,12pt]{amsart}
\usepackage[margin=1.15in]{geometry}

\usepackage{amssymb, enumerate}
\usepackage{mathrsfs}
\usepackage[all]{xy}
\usepackage{hyperref}
\hypersetup{colorlinks=true}
\usepackage{enumitem}

\theoremstyle{plain}
\newtheorem{thm}{Theorem}[section] 
\newtheorem{cor}[thm]{Corollary}
\newtheorem{prop}[thm]{Proposition}

\newtheorem{lem}[thm]{Lemma}
\newtheorem*{mainthm}{Theorem A}
\newtheorem*{mainthm2}{Theorem B}

\theoremstyle{definition} 
\newtheorem{defn}[thm]{Definition}
\newtheorem{setting}[thm]{Setting}
\newtheorem{eg}[thm]{Example} 

\newtheorem*{notation}{Notation}

\theoremstyle{remark}
\newtheorem{rem}[thm]{Remark}

\newtheorem*{claim}{Claim}
\newtheorem*{acknowledgement}{Acknowledgments}

\def\ge{\geqslant}
\def\le{\leqslant}
\def\phi{\varphi}
\def\epsilon{\varepsilon}
\def\tilde{\widetilde}

\newcommand{\sO}{\mathcal{O}}

\newcommand{\F}{\mathbb{F}}
\newcommand{\N}{\mathbb{N}}
\newcommand{\Q}{\mathbb{Q}} 
\newcommand{\C}{\mathbb{C}} 
 
\newcommand{\Z}{\mathbb{Z}}

\newcommand{\ba}{\mathfrak{a}}
\newcommand{\bb}{\mathfrak{b}}
\newcommand{\m}{\mathfrak{m}}

\newcommand{\q}{\mathfrak{q}}

\newsavebox{\circlebox}
\savebox{\circlebox}{\fontencoding{OMS}\selectfont\Large\char13}
\newlength{\circleboxwdht}

\def\Hom{\operatorname{Hom}}
\def\Spec{\operatorname{Spec}}
\def\Proj{\operatorname{Proj}}
\def\Supp{\operatorname{Supp}}

\def\Div{\operatorname{div}}

\def\ord{\operatorname{ord}}

\title{Arithmetic and geometric deformations of\\ $F$-pure and $F$-regular singularities}

\author{Kenta Sato}
\address{Faculty of Mathematics, Kyushu University, 744 Motooka, Nishi-ku, Fukuoka 819-0395, Japan}
\email{ksato@math.kyushu-u.ac.jp}

\author{Shunsuke Takagi}
\address{Graduate School of Mathematical Sciences, University of Tokyo, 3-8-1 Komaba, Meguro-ku, Tokyo 153-8914, Japan}
\email{stakagi@ms.u-tokyo.ac.jp}



\begin{document}

\begin{abstract}
Given a normal $\Q$-Gorenstein complex variety $X$, we prove that if one spreads it out to a normal $\Q$-Gorenstein scheme $\mathcal{X}$ of mixed characteristic whose reduction $\mathcal{X}_p$ modulo $p$ has normal $F$-pure singularities for a single prime $p$, then $X$ has log canonical singularities. 
In addition, we show its analog for log terminal singularities, without assuming that $\mathcal{X}$ is $\Q$-Gorenstein, which is a generalization of a result of Ma-Schwede. 
We also prove that two-dimensional strongly $F$-regular singularities are stable under equal characteristic deformations. 
Our results give an affirmative answer to a conjecture of Liedtke-Martin-Matsumoto on deformations of linearly reductive quotient singularities. 
\end{abstract}

\maketitle
\markboth{K.~SATO and S.~TAKAGI}{DEFORMATIONS OF $F$-PURE AND $F$-REGULAR SINGULARITIES}

\section{Introduction}

$F$-singularities are singularities in positive characteristic defined via the Frobenius morphism. 
The link between $F$-singularities and singularities in  minimal model program has been studied intensively over the last two decades. 
To explain this link, we establish some notation. 
Let $(x \in X)$ be a normal $\Q$-Gorenstein singularity over $\C$, that is, 
$X=\Spec \C[X_1, \dots, X_r]/(f_1, \dots, f_s)$ 
is a normal $\Q$-Gorenstein affine variety over $\C$ and $x$ is a closed point of $X$. 
We assume for simplicity that $x$ corresponds to the maximal ideal $(X_1, \dots, X_r)$ and all the coefficients of the $f_i$ are rational numbers. 
Choose an integer $n \ge 1$ such that the $f_i$ are polynomials over $\Z[1/n]$ and $\mathcal{X}:=\Spec \Z[1/n][X_1, \dots, X_r]/(f_1, \dots, f_s)$ is flat over $U:=\Spec \Z[1/n] \subset \Spec \Z$. 
Let $\mathcal{Z}$ be the closed subscheme of $\mathcal{X}$ defined by $(X_1, \dots, X_r)$ and $x_{\eta}$ (resp.~$x_p$) be the unique point of the generic fiber $\mathcal{Z}_{\eta}$ (resp.~the fiber $\mathcal{Z}_p$ over each closed point $(p) \in U$) of the flat morphism $\mathcal{Z} \subseteq \mathcal{X} \to U$.  
Then $(x \in X)$ is a flat base change of the generic fiber $(x_{\eta} \in \mathcal{X}_{\eta})$,  
while the fiber $(x_p \in \mathcal{X}_p)$
over each closed point $(p) \in U$ is of characteristic $p>0$. 
The link mentioned above is provided by comparing the properties of $(x \in X)$ 
with those of the closed fibers $(x_p \in X_p)$. 
In this paper, pursuing this link, we study how mild the singularity $(x \in X)$ is when some closed fiber $(x_p \in X_p)$ is a mild $F$-singularity such as $F$-pure and strongly $F$-regular singularities. 

Hara-Watanabe \cite{HW} proved that $(x \in X)$ is a log terminal (resp.~log canonical) singularity if the closed fiber $(x_p \in \mathcal{X}_p)$ is strongly $F$-regular (resp.~$F$-pure) for infinitely many closed points $(p) \in U$. 
Recently, using perfectoid techniques, Ma-Schwede \cite{MS} developed a new theory of singularities in mixed characteristic, where BCM test ideals, a generalization of test ideals to mixed characteristic, play a central role. 
As an application of this theory, they proved, under the assumption that the total space $\mathcal{X}$ is $\Q$-Gorenstein,  
that $(x \in X)$ is log terminal if the closed fiber $(x_p \in \mathcal{X}_p)$ is strongly $F$-regular for a single closed point $(p) \in U$. 
In this paper, we generalize their result to the case where the total space $\mathcal{X}$ is not necessarily $\Q$-Gorenstein.
We also prove the analog of their result for log canonical singularities. 
A simple form of our first main result is stated (with renewed notation) as follows: 

\begin{mainthm}[\textup{cf.~Corollaries \ref{mixed lc local} and \ref{mixed klt local}}]
 Let $(\mathcal{X}, \mathcal{D}) \to U \subseteq \Spec \Z$ be a  flat family of pairs, where $\mathcal{X}$ is a normal integral scheme and $\mathcal{D}$ is an effective $\Q$-Weil divisor on $\mathcal{X}$. 
 Let $\mathcal{Z}$ be an irreducible closed subscheme of $\mathcal{X}$, flat over $U$, such that the fiber of $\mathcal{Z} \subseteq \mathcal{X} \to U$ over each closed point $(p) \in U$ is a singleton $\{x_p\}$ and let $x_{\eta} \in \mathcal{X}_{\eta}$ 
 be the generic point of $\mathcal{Z}$.     
\begin{enumerate}[label=$(\arabic*)$]
\item 
Suppose that the generic fiber $(\mathcal{X}_{\eta}, \mathcal{D}_{\eta})$ is log $\Q$-Gorenstein at $x_{\eta}$, that is, $K_{\mathcal{X}_{\eta}}+\mathcal{D}_\eta$ is $\Q$-Cartier at $x_\eta$.
If the fiber $(\mathcal{X}_p, \mathcal{D}_p)$ is log $\Q$-Gorenstein and strongly $F$-regular at $x_p$ for a single closed point $(p) \in U$, then $(\mathcal{X}_{\eta}, \mathcal{D}_{\eta})$ is klt at $x_{\eta}$. 
\item
If the total space $(\mathcal{X}, \mathcal{D})$ is log $\Q$-Gorenstein at $x_p$ and the fiber $(\mathcal{X}_p,  \mathcal{D}_p)$ is normal and (sharply) $F$-pure at $x_p$ for a single closed point $(p) \in U$, then the generic fiber $(\mathcal{X}_{\eta}, \mathcal{D}_{\eta})$ is log canonical at $x_{\eta}$. 
\end{enumerate}
\end{mainthm}

This theorem is quite useful for verifying that a given singularity is log terminal or log canonical, because strong $F$-regularity and $F$-purity can be relatively easily checked by a computer algebra system such as Macaulay 2, compared with constructing a resolution of singularities (see Remark \ref{algorithm}). 

We briefly explain the idea of the proof of Theorem A. 
First we generalize the notion of the BCM test ideal $\tau_{\mathcal{B}}(R, \ba^\lambda)$, which was defined in \cite{MS} when $\ba$ is a principal ideal, to the case of an arbitrary ideal $\ba$. 
Replacing multiplier ideals by our BCM test ideals, one can employ the arguments of Kawakita \cite{Kaw}, which prove that log canonical singularities over $\C$ deform to log canonical singularities if the total space is $\Q$-Gorenstein, to obtain the assertion (2). 
Similarly, one can use arguments analogous to those of Esnault-Viehweg \cite{EV}, which prove that two-dimensional log terminal singularities over $\C$ deform to log terminal singularities, to obtain the assertion (1). 

Theorem A concerns arithmetic deformations of $F$-singularities, and next we discuss geometric deformations of $F$-singularities. 
In general, log terminal singularities over $\C$ are not stable under small deformations unless the total space is $\Q$-Gorenstein. 
A notable exception is the two-dimensional case, settled by \cite{EV} as mentioned above.  
In fact, their proof tells us that higher-dimensional log terminal singularities over $\C$ are also stable under small deformations if the nearby fibers are $\Q$-Gorenstein. 
Since strongly $F$-regular singularities can be viewed as an $F$-singularity theoretic counterpart of log terminal singularities, it is natural to ask how strongly $F$-regular singularities behave under equal characteristic deformations. 
Strong $F$-regularity does not deform in general (see \cite{Sin}), but 
we prove that the analog of the result of Esnault-Viehweg holds for strongly $F$-regular singularities. 
A simple form of our second main result is stated as follows: 

\begin{mainthm2}[\textup{cf.~Corollaries \ref{mixed klt proper} and \ref{general SFR fibers}}]
Let $(\mathcal{X}, \mathcal{D}) \to T$ be a proper flat family of pairs, where $\mathcal{X}$ is a normal variety over a perfect field $k$ of characteristic $p>0$,  
$\mathcal{D}$ is an effective $\Q$-Weil divisor on $\mathcal{X}$, and $T$ is a smooth curve over $k$. 
\begin{enumerate}[label=$(\arabic*)$]
\item
Suppose that the generic fiber $(\mathcal{X}_{\eta}, \mathcal{D}_\eta)$ and some closed fiber $(\mathcal{X}_{t_0}, \mathcal{D}_{t_0})$ are log $\Q$-Gorenstein. 
If $(\mathcal{X}_{t_0}, \mathcal{D}_{t_0})$ is strongly $F$-regular,  
then so is the geometric generic fiber $(\mathcal{X}_{\overline{\eta}}, \mathcal{D}_{\overline{\eta}})$. 
\item
Suppose that a general closed fiber $(\mathcal{X}_{t}, \mathcal{D}_t)$ and some closed fiber $(\mathcal{X}_{t_0}, \mathcal{D}_{t_0})$ are log $\Q$-Gorenstein. 
If $k$ is an uncountable algebraically closed field and $(\mathcal{X}_{t_0}, \mathcal{D}_{t_0})$ is strongly $F$-regular,  
then so is the general closed fiber $(\mathcal{X}_{t}, \mathcal{D}_t)$. 
\end{enumerate}
\end{mainthm2}
As its corollary, we see that two-dimensional strongly $F$-regular singularities are stable under equal characteristic deformations. 
The proof of Theorem B is similar to that of Theorem A (2), but we use classical test ideals instead of BCM test ideals. 

Finally, Theorems A and B enable us to give an affirmative answer to a conjecture of Liedtke-Martin-Matsumoto \cite[Conjecture 12.1 (1)]{LMM}, which states that isolated linearly reductive quotient singularities deform to linearly reductive quotient singularities in arbitrary characteristic. 

\begin{acknowledgement}
The authors are grateful to Linquan Ma, Yuya Matsumoto, and Karl Schwede for their helpful comments on a preliminary version of this paper. 
The second author would like to thank Paolo Cascini and Hiromu Tanaka for fruitful discussions. 
He also would like to express his gratitude to Imperial College London, where a part of this work was done, for their hospitality during the summer of 2018. 
The first author was partially supported by RIKEN iTHEMS Program and JSPS KAKENHI Grant Number 20K14303.
The second author was partially supported by JSPS KAKENHI Grant Numbers JP15H03611, JS15KK0152, JP16H02141, and JP17H02831. 
\end{acknowledgement}

\begin{notation}
Throughout this paper, all rings are assumed to be commutative and with unit element and all schemes are assumed to be Noetherian and separated. 
\end{notation}

\section{Preliminaries}

\subsection{Singularities in MMP}
In this subsection, we recall the definition and basic properties of singularities in minimal model program (MMP, for short). 

Throughout this subsection, we assume that $X$ is an excellent normal integral scheme with dualizing complex $\omega_X^\bullet$. 
The \textit{canonical sheaf} $\omega_X$ associated to $\omega_X^{\bullet}$ is the coherent $\sO_X$-module defined as the first nonzero cohomology module of $\omega_X^\bullet$.
A \textit{canonical divisor} of $X$ associated to $\omega_X^\bullet$ is any Weil divisor $K_X$ on $X$ such that $\sO_X(K_X) \cong \omega_X$. 
We fix a canonical divisor $K_X$ of $X$ associated to $\omega_X^\bullet$. 
Given a proper birational morphism $\pi:Y \to X$ from a normal integral scheme $Y$, we always choose a canonical divisor $K_Y$ of $Y$ which is associated to $\pi^! \omega_X^{\bullet}$ and whose pushforward $f_*K_Y$ agrees with $K_X$. 

\begin{defn}
A proper birational morphism $f: Y \to X$ between integral schemes is said to be a \textit{resolution of singularities} 
of $X$ if $Y$ is regular. 
When $\Delta$ is a $\Q$-Weil divisor on $X$ and $\ba, \bb\subseteq \sO_X$ are nonzero coherent ideal sheaves,  
a resolution $f:Y \to X$ is said to be a \textit{log resolution} of $(X, \Delta, \ba, \bb)$ if $\ba \sO_Y = \sO_Y(-F)$ and $\bb \sO_Y=\sO_Y(-G)$ are invertible 
and if the union of the exceptional locus $\mathrm{Exc}(f)$ of $f$ and the supports of $F$, $G$ and the strict transform $f^{-1}_*\Delta$ of $\Delta$ is a simple normal crossing divisor.  
\end{defn}

First we give the definition of singularities in MMP that makes sense in arbitrary characteristic. 
\begin{defn}\label{mmp}
Suppose that 
$\Delta$ is an effective $\Q$-Weil divisor on $X$ such that $K_X+\Delta$ is $\Q$-Cartier, $\ba \subseteq \sO_X$ is a nonzero coherent ideal sheaf, and $\lambda > 0$ is a real number. 
\begin{enumerate}[label=(\roman*)]
\item 
Given a proper birational morphism $f:Y \to X$ from a normal integral scheme $Y$, we write  
\[
\Delta_Y : = f^*(K_X+\Delta)-K_Y.
\]
When $\ba \sO_Y=\sO_Y(-F)$ is invertible, for each prime divisor $E$ on $Y$, the \textit{discrepancy} $a_{E}(X, \Delta, \ba^\lambda)$ of the triple $(X, \Delta, \ba^\lambda)$ at $E$ is defined as 
\[
a_{E} (X, \Delta, \ba^\lambda) : = - \ord_E( \Delta_Y + \lambda F). 
\] 
\item 
The triple $(X, \Delta, \ba^\lambda)$ is said to be \textit{log canonical} (resp.~\textit{klt}) at a point $x \in X$ if $a_E(\Spec \sO_{X,x}, \Delta_x, \ba_x^{\lambda}) \ge -1$ (resp.~$ >-1$) for every proper birational morphism $f:Y \to \Spec \sO_{X,x}$ from a normal integral scheme $Y$ with $\ba \sO_Y$ invertible and for every prime divisor $E$ on $Y$, 
where $\Delta_x$ is the flat pullback of $\Delta$ by the canonical morphism $\Spec \sO_{X,x} \to X$ and $\ba_x:=\ba \sO_{X,x}$. 
We say that $(X, \Delta, \ba^\lambda)$ is \textit{log canonical} (resp.~\textit{klt}) if it is log canonical (resp.~klt) for every $x \in X$. 
\item (\cite{LT}) $X$ is said to be \textit{pseudo-rational} at a point $x \in X$ if $X$ is Cohen-Macaulay and if for every projective birational morphism $f: Y \to \Spec \sO_{X,x}$ from a normal integral scheme $Y$, the natural morphism $f_* \omega_Y \to \omega_{X,x}$ is an isomorphism. 
\end{enumerate}
\end{defn}

\begin{rem}\label{remark rat}
\ 
\begin{enumerate}[label=(\roman*)]
\item 
(\cite[Theorem 1.4]{Kov}) If $X$ is pseudo-rational at $x$ and $f: Y \to \Spec \sO_{X,x}$ is a resolution of singularities, then $R^i f_*\sO_Y=0$ for all $i>0$.

\item (\cite[Proposition 17.1]{Lip}) If $X$ is pseudo-rational at $x$ and $\dim \sO_{X,x} =2$, then $X$ is $\Q$-factorial at $x$.

\item (\cite[Theorem 5.22]{KM}) Let $\Delta$ be an effective $\Q$-Weil divisor on $X$ such that $K_X+\Delta$ is $\Q$-Cartier.
If $X$ is essentially of finite type over a field of characteristic zero and $(X,\Delta)$ is klt at $x$, then $X$ is pseudo-rational at $x$. 
\end{enumerate}
\end{rem}

\begin{defn}
Let $(X,\Delta, \ba^\lambda)$ be as in Definition \ref{mmp}.
The \textit{multiplier ideal sheaf} $\mathcal{J}(X,\Delta,\ba^\lambda)$ associated to $(X,\Delta, \ba^\lambda)$ is defined as  
\[
\mathcal{J}(X,\Delta,\ba^\lambda) : = \bigcap_{f: Y \to X} f_*\sO_Y(- \lfloor \Delta_Y + \lambda F \rfloor),
\]
where $f: Y \to X$ runs through all proper birational morphisms from a normal integral scheme $Y$ with $\ba \sO_Y= \sO_Y(- F)$ invertible and $\Delta_Y := f^*(K_X+\Delta) -K_Y$. 
\end{defn}

\begin{rem}\label{resolution}
Let $(X,\Delta, \ba^\lambda)$ be as in Definition \ref{mmp}.
\begin{enumerate}[label=(\roman*)]
\item For any point $x \in X$, we have $\mathcal{J}(X,\Delta, \ba^\lambda)_x \subseteq \mathcal{J}(\Spec \sO_{X,x}, \Delta_x, \ba_x^\lambda)$.
In particular, if $\mathcal{J}(X,\Delta, \ba^\lambda)_x = \sO_{X,x}$, then $(X,\Delta, \ba^\lambda)$ is klt at $x$.

\item If $f: Y \to X$ is a log resolution of $(X,\Delta,\ba)$, then 
\[
\mathcal{J}(X,\Delta,\ba^\lambda) = f_* \sO_Y(-\lfloor \Delta_Y+ \lambda F \rfloor).
\]
In particular, if $X$ is defined over a field of characteristic zero or $\dim X \le 2$, then the following hold:
\begin{enumerate}[label=(\alph*)]
\item the converse of (i) is true, 
\item $\mathcal{J}(X,\Delta, \ba^\lambda)$ is coherent, 
\item it is enough to check the condition in Definition \ref{mmp} (ii) only for one $f$, namely, for a log resolution of $(X, \Delta, \ba^\lambda)$. 
\end{enumerate}
\end{enumerate}
\end{rem}

The following lemma is probably well-known to experts, but we include a proof for the reader's convenience. 
\begin{lem}\label{regular map and lc}
Suppose that $(R,\m, \kappa) \to (R' ,\m',\kappa')$ is a flat local homomorphism of excellent local rings with dualizing complexes.
Assume that we have $\m R' =\m'$ and $\kappa'$ is separable over $\kappa$.
Let $f: X' : = \Spec R' \to X : = \Spec R$ be the induced morphism sending a point $x' \in X'$ to a point $x \in X$, $\Delta$ be an effective $\Q$-Weil divisor on $X : = \Spec R$ such that $K_X+\Delta$ is $\Q$-Cartier at $x$, $\ba \subseteq R$ be a nonzero ideal sheaf and $\lambda > 0$ be a real number.
\begin{enumerate}[label=$(\arabic*)$]
\item The flat pullback $K_{X'} : = f^*K_X$ of $K_X$ by the induced morphism $f: X' : = \Spec R' \to X$ is a canonical divisor of $X'$. 
In particular, $K_{X'} +\Delta'$ is $\Q$-Cartier at $x'$, where $\Delta' := f^*\Delta$ is the flat pullback of $\Delta$.  
\item If $(X', \Delta', (\ba \sO_{X'})^\lambda)$ is log canonical $($resp.~klt$)$ at $x'$, then so is $(X, \Delta, \ba^\lambda)$ at $x$.
\item If $X$ is defined over a field of characteristic zero, then the converse of $\textup{(2)}$ also holds. 
\end{enumerate}
\end{lem}

\begin{proof}
(1) First note by \cite[Lemma 0AWD]{Sta} that $\omega_{X'}^{\bullet} : = f^* \omega_X^{\bullet}$ is a dualizing complex on $X'$. 
Since $f^* \sO_X(K_X) \cong \sO_{X'}(f^*K_{X})$, we see that $f^* K_X$ is a canonical divisor of $X'$ associated to $\omega_{X'}^{\bullet}$. 

(2) Since the closed fiber $\Spec \kappa'$ of $f$ is formally smooth over $\Spec \kappa$, it follows from \cite[Theorem 28.10]{Mat} and \cite{And} that $f$ is a regular morphism, that is, all fibers are geometrically regular.
We set $V: = \Spec \sO_{X,x}$ and $V' : = \Spec \sO_{X',x'}$.
Then the induced morphism $h : V' \to V$ is regular, $\omega_V^{\bullet} : = \omega_X^{\bullet}|_V$ is a dualizing complex of $V$ and $\omega_{V'}^{\bullet} : = \omega_{X'}^{\bullet}|_{V'} \cong h^{*} \omega_V^{\bullet}$ is a dualizing complex of $V'$.

Let $\pi:Y \to V$ be a proper birational morphism from a normal integral scheme $Y$ with $\ba \sO_Y=\sO_Y(-F)$ invertible. 
Set $Y':= V' \times_V Y$ and let $\pi' : Y' \to V'$ and $g: Y' \to Y$ 
be the first and second projections, respectively. 

\begin{claim} 
  $Y'$ is a normal integral scheme. 
\end{claim}
\begin{proof}
Let $K$ be a function field of $V$ and $Y$. First note that
\[ 
Y' \times_{Y} \Spec K \cong V' \times_V \Spec K,
\]
because $\pi$ is birational. 
Since $V' \times_V \Spec K$ is irreducible, so is $Y' \times_{Y} \Spec K$.
Taking into account that any generic point of $Y'$ lies in the generic fiber $Y' \times_{Y} \Spec K$ of $g$, we see that $Y'$ is irreducible. 
It then follows from \cite[p.184 Corollary]{Mat} that $Y'$ is a normal integral scheme. 
\end{proof}

Fix a canonical divisor $K_Y$ of $Y$ associated to the dualizing complex $\omega_Y^{\bullet} : = \pi^{!} \omega_V^{\bullet}$ such that $\pi_* K_Y =K_V$.
We write $\Delta_Y := \pi^*(K_V+\Delta|_V) - K_Y$, and let $F' : = g^*F$, $\Delta_{Y'} := g^* \Delta_Y$ and $K_{Y'} : = g^* K_Y$ be the flat pullbacks of $F$, $\Delta_Y$ and $K_Y$, respectively. 
Then $\ba  \sO_{Y'}=\sO_{Y'}(-F')$ and $\Delta_{Y'} = {\pi'}^{*}(K_V' + \Delta') - K_{Y'}$.
Noting that $K_{Y'}$ is a canonical divisor of $Y'$ associated to ${\pi'}^{!} ( \omega_{V'}^{\bullet}) \cong g^* \omega_Y^{\bullet}$ (see \cite[Lemma 0AA8]{Sta}) such that ${\pi'}_*K_{Y'} = K_{X'}$, we deduce  
\[
-\ord_G ( g^* (\Delta_Y + \lambda F))
=- \ord_G( \Delta_{Y'} + \lambda F')
=a_{G}(V',\Delta'|_{V'}, (\ba \sO_{V'})^\lambda)
\]
for every prime divisor $G$ on $Y'$. 
Since $g$ is regular, the flat pullback $g^* E$ of a prime divisor $E$ on $Y$ is a reduced divisor on $Y'$ (see \cite[p.184 Corollary]{Mat} again).
Therefore, for any irreducible component $E'$ of $g^{-1}(E)$, we have 
\begin{align*}
 a_E(V,\Delta|_V, (\ba \sO_V)^\lambda)=-\ord_{E} (\Delta_Y + \lambda F)&=-\ord_{E'} (g^* (\Delta_Y + \lambda F))\\
 &=a_{E'}(V', \Delta'|_{V'}, (\ba \sO_{V'})^\lambda),
\end{align*}
which proves (2).

(3) We choose a log resolution of $(V, \Delta|_V ,\ba \sO_V)$ to be $\pi$ in the proof of (2). 
It then follows from \cite[Theorem 23.7]{Mat} that $\pi'$ is also a log resolution of $(V', \Delta'|_{V'} ,\ba \sO_{V'})$.
Therefore, the assertion is immediate from Remark \ref{resolution}.
\end{proof}

\subsection{\texorpdfstring{$F$}{F}-singularities}
We briefly review the theory of $F$-singularities, especially focusing on strongly $F$-regular and $F$-pure singularities and test ideals. 

\begin{defn}[\textup{\cite{HW}, \cite{Tak1}, \cite{Sch0}, \cite{Sch1}}]\label{F-sing def}
Let $x$ be a point of an $F$-finite normal integral scheme $X$ and $\Delta$ be an effective $\Q$-Weil divisor on $X$.
Given an integer $e \ge 1$, let 
\[\varphi^{(e)}_{\Delta}: \sO_X \to F^e_*\sO_X \hookrightarrow F^e_*\sO_X(\lceil (p^e-1)\Delta \rceil)\]
be  
the composite of the $e$-times iterated Frobenius map $\sO_X \to F^e_*\sO_X$ and the pushforward 
of the natural inclusion $\sO_X \hookrightarrow \sO_X(\lceil (p^e-1)\Delta \rceil)$ by $F^e$. 
Let $\ba_1, \dots, \ba_\ell \subseteq \sO_X$ be nonzero coherent ideal sheaves and $\lambda_1, \dots, \lambda_\ell \ge 0$ be real numbers. 
\begin{enumerate}[label=\textup{(\roman*)}]
\item $(X,\Delta, \ba_{1}^{\lambda_1} \cdots \ba_{\ell}^{\lambda_\ell})$ is said to be \textit{sharply $F$-pure} at $x$ if there exist an integer $e \ge 1$ and a nonzero element $d \in \ba_{1}^{\lceil \lambda_1(p^e-1)\rceil} \cdots \ba_\ell^{\lceil \lambda_{\ell}(p^e-1)\rceil}\sO_{X,x}$ such that the composite 
\[
  \sO_{X,x} \xrightarrow{\varphi^{(e)}_{\Delta,x}} F^e_*\sO_X(\lceil (p^e-1)\Delta \rceil)_x \xrightarrow{\times F^e_*d} F^e_*\sO_X(\lceil (p^e-1)\Delta \rceil)_x
\] 
of the $\sO_{X,x}$-linear map $\varphi^{(e)}_{\Delta,x}$ induced by $\varphi^{(e)}_{\Delta}$ and the multiplication map by $F^e_*d$ splits as an $\sO_{X,x}$-module homomorphism. 

\item $(X,\Delta, \ba_{1}^{\lambda_1} \cdots \ba_{\ell}^{\lambda_\ell})$ is said to be \textit{strongly $F$-regular} at $x$ if for every nonzero element $c \in \sO_{X,x}$, there exist $e \ge 1$ and $0 \ne d \in \ba_1^{\lceil \lambda_1(p^e-1)\rceil} \cdots \ba_\ell^{\lceil \lambda_{\ell}(p^e-1)\rceil}\sO_{X,x}$ such that the composite 
\[
  \sO_{X,x} \xrightarrow{\varphi^{(e)}_{\Delta,x}} F^e_*\sO_X(\lceil (p^e-1)\Delta \rceil)_x \xrightarrow{\times F^e_*(cd)} F^e_*\sO_X(\lceil (p^e-1)\Delta \rceil)_x
\] 
of the $\sO_{X,x}$-linear map $\varphi^{(e)}_{\Delta,x}$ induced by $\varphi^{(e)}_{\Delta}$ and the multiplication map by $F^e_*(cd)$ splits as an $\sO_{X,x}$-module homomorphism. 
\end{enumerate}

We say that $(X,\Delta, \ba_{1}^{\lambda_1} \cdots \ba_{\ell}^{\lambda_\ell})$ is sharply $F$-pure (resp. strongly $F$-regular) if it is sharply $F$-pure (resp. strongly $F$-regular) at all points of $X$.  
\end{defn}

\begin{rem}\label{sharp F-pure remark}
It is known by \cite[Lemma 2.8]{Sch1} that if $(X,\Delta, \ba_{1}^{\lambda_1} \cdots \ba_{\ell}^{\lambda_\ell})$ is sharply $F$-pure at $x$, then there exist infinitely many integer $e \ge 1$ satisfying the condition in Definition \ref{F-sing def} (i). 
\end{rem}

\begin{rem}[$\textup{\cite{FW}, \cite[Theorem 3.1]{Smi}}$]\label{SFR to F-rat}
We have the following hierarchy of properties of $F$-finite normal singularities (see \cite{FW} for the definition of $F$-rational singularities): 
\[
\textup{strongly $F$-regular} \Longrightarrow \textup{$F$-rational} \Longrightarrow \textup{pseudo-rational}.  
\]
\end{rem}

\begin{defn}[\textup{\cite[Definition-Proposition 3.3]{BSTZ}, cf.~\cite{HY}, \cite{Tak0}}]\label{test ideal def}
Let $(R, \m)$ be an $F$-finite normal local ring of characteristic $p>0$ and $\Delta$ be an effective $\Q$-Weil divisor on $X:=\Spec R$. 
Let $\ba_1, \dots, \ba_\ell \subseteq R$ be nonzero ideals and $\lambda_1, \dots, \lambda_\ell \ge 0$ be real numbers. 
Fix a big sharp test element $d \in R$ for $(X, \Delta, \ba_{1}^{\lambda_1} \cdots \ba_{\ell}^{\lambda_\ell})$ (see \cite[Definition 2.16]{Sch1} for the definition of big sharp test elements). 
Then 
the \textit{test ideal} $\tau(X, \Delta, \ba_{1}^{\lambda_1} \cdots \ba_{\ell}^{\lambda_\ell})$\footnote{This ideal is often refereed as the ``big test ideal" for $(X, \Delta, \ba_{1}^{\lambda_1} \cdots \ba_{\ell}^{\lambda_\ell})$ and denoted by $\widetilde{\tau}(X, \Delta, \ba_{1}^{\lambda_1} \cdots \ba_{\ell}^{\lambda_\ell})$ or $\tau_b(X, \Delta, \ba_{1}^{\lambda_1} \cdots \ba_{\ell}^{\lambda_\ell})$  in the literature.}  for the triple $(X, \Delta, \ba_{1}^{\lambda_1} \cdots \ba_{\ell}^{\lambda_\ell})$ is defined by 
\[
  \tau(X, \Delta, \ba_{1}^{\lambda_1} \cdots \ba_{\ell}^{\lambda_\ell})=\sum_{e \ge 0} \sum_{\psi} \psi(F^e_*(d \ba_1^{\lceil \lambda_1(p^e-1) \rceil} \cdots \ba_\ell^{\lceil \lambda_\ell(p^e-1)\rceil})),  
\]
where $e$ runs through all nonnegative integers and $\psi$ runs through all elements of $\Hom_R(F^e_*R(\lceil (p^e-1)\Delta \rceil), R)$. 
\end{defn}

\begin{rem}\label{test element remark}
We do not give the definition of big sharp test elements in this paper, but such elements always exists by \cite[Lemma 2.17]{Sch1}. 
It follows from an argument analogous to the proof of the equivalence of (4) and (5) in \cite[Definition-Proposition 3.3]{BSTZ} that 
if $d'$ is a big sharp test element for $(X, \Delta, \ba_{1}^{\lambda_1} \cdots \ba_{\ell -1}^{\lambda_{\ell-1}})$, then 
\[
  \tau(X, \Delta, \ba_{1}^{\lambda_1} \cdots \ba_{\ell}^{\lambda_\ell})=\sum_{e \ge 0} \sum_{\psi} \psi(F^e_*(d' \ba_1^{\lceil \lambda_1(p^e-1) \rceil} \cdots \ba_{\ell-1}^{\lceil \lambda_{\ell-1}(p^e-1)\rceil}\ba_{\ell}^{\lceil \lambda_{\ell}p^e \rceil})),  
\]
where $\psi$ runs through all elements of $\Hom_R(F^e_*R(\lceil (p^e-1)\Delta \rceil), R)$. 
Similarly, if $c$ is a big sharp test element for $(X,\Delta)$, then 
\[
  \tau(X, \Delta, \ba_{1}^{\lambda_1} \cdots \ba_{\ell}^{\lambda_\ell})=\sum_{e \ge 0} \sum_{\psi} \psi(F^e_*(c \ba_1^{\lceil \lambda_1 p^e \rceil} \cdots \ba_{\ell}^{\lceil \lambda_{\ell}p^e \rceil})),  
\]
where $\psi$ runs through all elements of $\Hom_R(F^e_*R(\lceil (p^e-1)\Delta \rceil), R)$. 
\end{rem}

Test ideals for triples can be described as a sum of test ideals for pairs.  
\begin{lem}\label{test module for ideals}
Let the notation be the same as in Definition \ref{test ideal def}. 
Then 
\[
\tau(X, \Delta, \ba_1^{\lambda_1} \cdots \ba_\ell^{\lambda_\ell}) = \sum_{m_1, \dots, m_\ell \ge 1} \sum_{f_i \in \ba_i^{\lceil m_i \lambda_i \rceil}} \tau(X, \Delta + \frac{1}{m_1} \Div_X(f_1) + \cdots + \frac{1}{m_\ell} \Div_X(f_\ell)),
\]
where the first summation is taken over all integers $m_1, \dots, m_\ell \ge 1$ and the second summation is taken over all nonzero elements $f_i \in \ba_i^{\lceil m_i \lambda_i \rceil}$ for each $i=1, \dots, \ell$.
\end{lem}

\begin{proof}
$\tau'(X, \Delta, \ba_1^{\lambda_1} \cdots \ba_\ell^{\lambda_\ell})$ denotes the ideal on the right hand side. 
We first show that $\tau'(X, \Delta, \ba_1^{\lambda_1} \cdots \ba_\ell^{\lambda_\ell})$ in contained in $\tau(X, \Delta, \ba_1^{\lambda_1} \cdots \ba_\ell^{\lambda_\ell})$. 
When $m_i$ is a positive integer and $f_i \in \ba_i^{\lceil m_i \lambda_i \rceil}$ for each $i=1, \dots, \ell$, 
\begin{align*}
\tau(X, \Delta+\frac{1}{m_1} \Div_X(f_1) + \cdots + \frac{1}{m_\ell} \Div_X(f_\ell))&=\tau(X,\Delta, (f_1)^{1/m_1}\cdots (f_{\ell})^{1/m_\ell})\\
&\subseteq \tau(X, \Delta, \ba_1^{\lceil m_1 \lambda_1 \rceil/m_1} \cdots \ba_{\ell}^{\lceil m_\ell \lambda_\ell \rceil/m_\ell})\\
& \subseteq \tau(X, \Delta, \ba_1^{\lambda_1} \cdots \ba_\ell^{\lambda_\ell}).
\end{align*}
Thus, $\tau'(X, \Delta, \ba_1^{\lambda_1} \cdots \ba_\ell^{\lambda_\ell}) \subseteq \tau(X, \Delta, \ba_1^{\lambda_1} \cdots \ba_\ell^{\lambda_\ell})$. 

We next prove the reverse inclusion. 
Let $c \in R$ be a big sharp test element for $(R, \Delta)$. 
Then by Remark \ref{test element remark},  
\begin{align*}
\tau(X, \Delta, \ba_1^{\lambda_1} \cdots \ba_\ell^{\lambda_\ell})
&=\sum_{e \ge 0}\sum_{\psi}\psi(F^e_*(c\ba_1^{\lceil \lambda_1 p^e \rceil} \cdots \ba_{\ell}^{\lceil \lambda_\ell p^e \rceil}))\\
&=\sum_{e \ge 0} \sum_{g_i \in \ba_i^{\lceil \lambda_i p^e \rceil}} \sum_{\psi} \psi(F^e_*(cg_1 \cdots g_\ell))\\
&\subseteq \sum_{e \ge 0} \sum_{g_i \in \ba_i^{\lceil \lambda_i p^e \rceil}} \tau(X, \Delta, g_1^{1/p^e} \cdots g_{\ell}^{1/p^e})\\
&=\sum_{e \ge 0} \sum_{g_i \in \ba_i^{\lceil \lambda_i p^e \rceil}} \tau(X, \Delta+\frac{1}{p^e}\Div_X(g_1)+\cdots+\frac{1}{p^e}\Div_X(g_\ell)), 
\end{align*}
where  
$\psi$ ranges over all elements of $\Hom_R(F^e_*R(\lceil (p^e-1)\Delta \rceil), R)$, and $g_i$ ranges over all nonzero elements of $\ba_i^{\lceil \lambda_i p^e \rceil}$ for each $i=1, \dots, \ell$. 
Therefore, $\tau(X, \Delta, \ba_1^{\lambda_1} \cdots \ba_\ell^{\lambda_\ell}) \subseteq \tau'(X, \Delta, \ba_1^{\lambda_1} \cdots \ba_\ell^{\lambda_\ell})$. 
\end{proof}

\subsection{Deformations}
In this subsection, we recall some basic terminology from the theory of deformations.

\begin{defn}
Let $X$ be an algebraic scheme over a field $k$.  Suppose that $T$ is a scheme and $t \in T$ is a $k$-rational point. 
A \textit{deformation} of $X$ over $T$ with reference point $t$ is a pair $(\mathcal{X}, i)$ of a scheme $\mathcal{X}$ which is flat and of finite type over $T$ and an isomorphism $i:X  \xrightarrow{\ \sim\ }  \mathcal{X} \times_T \Spec \kappa(t)$ of $k$-schemes. 
\end{defn}

In the subsequent sections, we consider several problems on deformations of singularities with the following setup.

\begin{setting}\label{local setting}
Suppose that $X$ is a normal integral scheme over a field $k$,  $T$ is a regular integral scheme with generic point $\eta$ and $t \in T$ is a $k$-rational point. 
Let $(\mathcal{X},i)$ be a deformation of $X$ over $T$ with reference point $t$ such that $\mathcal{X}$ is an excellent normal integral scheme with dualizing complex. 
Let $\mathcal{D}$ be an effective $\Q$-Weil divisor on $\mathcal{X}$ whose support does not contain the closed fiber $X$. 
Let $\ba \subseteq \sO_{\mathcal{X}}$ be a coherent ideal sheaf such that $\ba \sO_X$ is nonzero and $\lambda > 0$ be a real number.
\end{setting}

\begin{rem}\label{Q-Gorenstein}
We use the notation in Setting \ref{local setting}. 
\begin{enumerate}[label=(\roman*)]
\item In Section 3, we mainly focus on the case where the following condition holds.
\begin{enumerate}
\item[(A)] the pair $(\mathcal{X}, \mathcal{D})$ on the total space $\mathcal{X}$ is log $\Q$-Gorenstein, that is, $K_{\mathcal{X}}+\mathcal{D}$ is $\Q$-Cartier.  
\end{enumerate}
The condition (A) implies the following two conditions:
\begin{enumerate}
\item[(B)] the pair $(\mathcal{X}_\eta, \mathcal{D}_\eta)$ on the generic fiber $\mathcal{X}_\eta$ is log $\Q$-Gorenstein, 
\item[(C)] the pair $(X, \mathcal{D}|_X)$ on the closed fiber $X$ is log $\Q$-Gorenstein.
\end{enumerate}
We note that there is no relation between (B) and (C) (see \cite[Example 9.1.8]{Is} or Example \ref{counter eg local} below for a counterexample to the implication (C) $\Rightarrow$ (B)).
On the other hand, in Sections 4 and 5, we discuss the case where the conditions (B) and (C) are satisfied though the condition (A) is not necessarily satisfied.

\item 
If $k$ is an uncountable algebraically closed field, $T$ is of finite type over $k$, and a general closed fiber of $\mathcal{X} \to T$ is normal, then the condition (B) is equivalent to the following condition:
\begin{enumerate}
\item[(D)] for a general closed point $s \in T$, the pair $(\mathcal{X}_s, \mathcal{D}_s)$ on the fiber $\mathcal{X}_s$ is log $\Q$-Gorenstein.
\end{enumerate}
Indeed, it is obvious that (B) implies (D). 
For the converse implication, we fix an integer $m>0$ such that $m \mathcal{D}$ is an integral Weil divisor. 
For every integer $n>0$, we consider the coherent sheaf $\mathcal{F}_n: =\sO_{\mathcal{X}}(n m( K_{\mathcal{X}}+\mathcal{D}))$.
Since $\mathcal{F}_n$ satisfies the $(S_2)$-condition, 
there exists an non-empty open subset $V_n \subseteq T$ such that for every point $s \in V_n$, the restriction $\mathcal{F}_n|_{\mathcal{X}_s}$ to the fiber $\mathcal{X}_s$ satisfies the $(S_2)$-condition. 
Therefore, for such $s$, we have 
\[
\mathcal{F}_n|_{\mathcal{X}_s} \cong \sO_{\mathcal{X}_s} (nm(K_{\mathcal{X}_s}+\mathcal{D}_s)),
\] 
and in particular, if $nm (K_{\mathcal{X}_s} + \mathcal{D}_s)$ is Cartier, then so is $nm (K_{\mathcal{X}} + \mathcal{D})$ along $\mathcal{X}_s$.

Since $k$ is uncountable, we find a closed point $s \in \bigcap_n V_n$ such that $(\mathcal{X}_s, \mathcal{D}_s)$ is log $\Q$-Gorenstein.
Now pick an integer $l>0$ such that $l m(K_{\mathcal{X}_s} +\mathcal{D}_s)$ is Cartier.
Taking into account that $s \in V_l$, we conclude that $l m(K_{\mathcal{X}} + \mathcal{D}) $ is Cartier along $\mathcal{X}_s$, which implies the condition (B). 

\item If $k$ is perfect and $\mathcal{X}$ is proper over $T$, then it follows from \cite[(12.2.4)]{EGA} that a general fiber of $\mathcal{X} \to T$ is geometrically normal, and in particular, the third assumption in (ii) is satisfied.

Similarly, if $k$ is perfect, then it follows from \cite[(12.1.6)]{EGA} that after shrinking $\mathcal{X}$ around $X$, we may assume that $\mathcal{X} \to T$ is a normal morphism, that is, all fibers are geometrically normal.
\end{enumerate}
\end{rem}

\section{Deformations with \texorpdfstring{$\Q$}{Q}-Gorenstein total space}

In this section, we study arithmetic deformations of $F$-pure singularities when the total space is $\Q$-Gorenstein. 

\subsection{BCM test ideals}
First we recall the definition of BCM test ideals for pairs introduced by Ma-Schwede \cite{MS}. 
\begin{defn}[\textup{\cite[Definition 6.2, Definition 6.9]{MS}}]
Let $(R,\m)$ be a $d$-dimensional complete normal local ring of mixed characteristic $(0,p)$ and fix an effective canonical divisor $K_X$ of $X:=\Spec R$. 
Let $\Delta$ be an effective $\Q$-Weil divisor on $X$ such that $K_X+\Delta$ is $\Q$-Cartier. 
Then there exist an integer $n \ge 1$ and nonzero element $f \in R$ such that $n(K_X+\Delta)=\Div_X(f)$. 
We set 
\[
  (0)^{\mathcal{B}, K_X+\Delta}_{H^d_\m(R)}:=\bigcup_{B} \ker \left( H^d_\m(R) \xrightarrow{\times f^{1/n}} H^d_\m(B)\right),
\]
where $B$ runs through all integral perfectoid big Cohen-Macaulay $R^+$-algebras. 
Then \textit{BCM test ideal} for $(X, \Delta)$ is defined as
\[
  \tau_{\mathcal{B}}(X, \Delta):=\mathrm{Ann}_{\omega_R}\, (0)^{\mathcal{B}, K_X+\Delta}_{H^d_\m(R)} \subseteq R.
\]
\end{defn}

We generalize the definition and some properties of BCM test ideals to the case of triples. 
\begin{defn}\label{BCM test ideal}
Let $R$ be a complete normal local ring of mixed characteristic $(0,p)$ and $\Delta$ be an effective $\Q$-Weil divisor on $X:=\Spec R$ such that $K_X+\Delta$ is $\Q$-Cartier. 
Let $\ba_1, \dots, \ba_l \subseteq R$ be nonzero ideals and $\lambda_1, \dots, \lambda_\ell \ge 0$ be real numbers. 
\begin{enumerate}[label=\textup{(\roman*)}]
\item We define the BCM-test ideal for the triple $(R, \Delta, \ba_1^{\lambda_1} \cdots \ba_\ell^{\lambda_\ell})$ as 
\[
\tau_{\mathcal{B}}(X, \Delta, \ba_1^{\lambda_1} \cdots \ba_\ell^{\lambda_\ell}): = \sum_{m_1, \dots, m_\ell} \sum_{f_i \in \ba^{\lceil m_i \lambda_i \rceil}} \tau_{\mathcal{B}}(X, \Delta + \frac{1}{m_1} \Div_X(f_1) + \cdots + \frac{1}{m_\ell} \Div_X(f_\ell)),
\]
where the first summation is taken over all positive integers $m_1, \dotsm m_\ell$ and the second summation is taken over all nonzero elements $f_i \in \ba^{\lceil m_i \lambda_i \rceil}$ for each $i=1, \dots, \ell$. 

\item We say that $(X, \Delta, \ba_1^{\lambda_1} \cdots \ba_{\ell}^{\lambda_\ell})$ is \emph{BCM-regular} if we have \[
\tau_{\mathcal{B}}(X, \Delta, \ba_1^{\lambda_1} \cdots \ba_\ell^{\lambda_\ell}) = R.
\]
\end{enumerate}
\end{defn}

\begin{lem}\label{tauB and birat}
Let the notation be the same as in Definition \ref{BCM test ideal}. 
Then we have 
\[
\tau_{\mathcal{B}}(X, \Delta, \ba_1^{\lambda_1} \cdots \ba_\ell^{\lambda_\ell}) \subseteq \mathcal{J}(X, \Delta, \ba_1^{\lambda_1} \cdots \ba_\ell^{\lambda_l}).
\]
\end{lem}

\begin{proof}
Let $\pi:Y \to X$ be a proper birational morphism from a normal integral scheme $Y$ such that $\ba_i \sO_Y=\sO_Y(-F_i)$ is invertible for every $i=1, \dots, \ell$. 
When $m_i$ is a positive integer and $f_i \in \ba_i^{\lceil m_i \lambda_i \rceil}$ for each $i=1, \dots, \ell$, it follows from \cite[Theorem 6.21]{MS} that 
\begin{align*}
&\tau_{\mathcal{B}}(X, \Delta + \frac{1}{m_1} \Div_X(f_1) + \cdots + \frac{1}{m_\ell} \Div_X(f_\ell))\\
\subseteq &\pi_*\sO_Y(\lceil K_Y-\pi^*(K_X+\Delta+\frac{1}{m_1} \Div_X(f_1) + \cdots + \frac{1}{m_\ell} \Div_X(f_\ell))\rceil)\\
\subseteq & \pi_*\sO_Y(\lceil K_Y-\pi^*(K_X+\Delta)-\lambda_1 F_1 - \cdots - \lambda_\ell F_\ell\rceil).
\end{align*}
Thus, $\tau_{\mathcal{B}}(X, \Delta, \ba_1^{\lambda_1} \cdots \ba_\ell^{\lambda_\ell}) \subseteq \mathcal{J}(X, \Delta, \ba_1^{\lambda_1} \cdots \ba_\ell^{\lambda_l})$. 
\end{proof}

\begin{lem}\label{restriction}
Let the notation be the same as in Definition \ref{BCM test ideal}. 
Let $h_1, \dots, h_r \in R$ be a regular sequence such that $S:=R/(h_1, \dots, h_r)$ is an $F$-finite normal local ring of characteristic $p$. 
In addition, We assume that $\ba_i$ is not contained in the ideal $(h_1, \dots, h_r)$ for all $i$ and $Z:=\Spec S$ is not contained in the support of $\Delta$. 
Then we have 
\[
\tau(Z,\Delta|_Z, (\ba_1 S)^{\lambda_1} \cdots (\ba_\ell S)^{\lambda_\ell}) \subseteq \tau_{\mathcal{B}}(X,\Delta, \ba_1^{\lambda_1} \cdots \ba_\ell^{\lambda_\ell}) S. 
\]
\end{lem}

\begin{proof}
If the Cartier index of $K_X+\Delta$ is not divisible by $p$, then the assertion follows from a combination of Lemma \ref{test module for ideals} and \cite[Theorem 6.27]{MS}. 
Therefore, we assume that the Cartier index $n$ of $K_X+\Delta$ is divisible by $p$. 
Choose an effective $\Q$-Weil divisor $D$ on $X$ that is linearly equivalent to $K_X+\Delta$ and does not contain $Z$ in its support. 
By \cite[Proposition 2.14 (2)]{Sat}, we can take a sufficiently large integer $s$ so that 
\[
\tau(Z,\Delta|_Z, (\ba_1 S)^{\lambda_1} \cdots (\ba_\ell S)^{\lambda_\ell})=\tau(Z,(\Delta+\Delta')|_Z, (\ba_1 S)^{\lambda_1} \cdots (\ba_\ell S)^{\lambda_\ell}),
\] 
where $\Delta':=\frac{n-1}{ns+1}D$.  
Since $(ns+1)(K_X+\Delta+\Delta') \sim n(s+1)(K_X+\Delta)$, the Cartier index of $K_X+\Delta+\Delta'$ is not divisible by $p$. 
It then follows from Lemma \ref{test module for ideals} and \cite[Theorem 6.27, Lemma 6.11]{MS} that 
\begin{align*}
\tau(Z,\Delta|_Z, (\ba_1 S)^{\lambda_1} \cdots (\ba_\ell S)^{\lambda_\ell})&=\tau(Z,(\Delta+\Delta')|_Z, (\ba_1 S)^{\lambda_1} \cdots (\ba_\ell S)^{\lambda_\ell})\\
&\subseteq \tau_{\mathcal{B}}(X,\Delta+\Delta', \ba_1^{\lambda_1} \cdots \ba_\ell^{\lambda_\ell})S\\
&\subseteq \tau_{\mathcal{B}}(X,\Delta, \ba_1^{\lambda_1} \cdots \ba_\ell^{\lambda_\ell})S. 
\end{align*}
\end{proof}

\subsection{Deformations of \texorpdfstring{$F$}{F}-pure singularities}

We start with two auxiliary lemmas on log canonical and $F$-pure singularities. 

\begin{lem}\label{non-lc}
Let $(R,\m)$ be a complete normal local ring 
and $\Delta$ be an effective $\Q$-Weil divisor on $X:=\Spec R$ such that $K_X+\Delta$ is $\Q$-Cartier. 
Let $\ba \subseteq R$ be a nonzero ideal and $\lambda > 0$ be a real number.
If $(X,\Delta, \ba^\lambda)$ is not log canonical, then there exist a descending chain of nonzero ideals of $R$
\[
R=\bb_0 \supseteq \bb_1 \supseteq \dots \supseteq \bb_n \supseteq \cdots  
\]
and a decreasing sequence of positive real numbers 
\[
1=\varepsilon_0 \ge \varepsilon_1 \ge \dots \ge \varepsilon_n \ge \dots  
\]
with following properties:  
\begin{enumerate}[label=\textup{(\roman*)}]
\item for every integer $n \ge 0$, we have
\[
\mathcal{J}(X, \Delta, \ba^\lambda \bb_n^{1-\varepsilon_n}) \subseteq \bb_{n+1},
\]
\item for every integer $\ell \ge 1$, there exists an integer $n(\ell) \ge 0$ such that $\bb_{n(\ell)} \subseteq \m^\ell$.
\end{enumerate}
\end{lem}

\begin{proof}
Let $f : Y \to X$ be a proper birational morphism from a normal connected scheme $Y$ such that $\ba \sO_Y =\sO_Y(-F)$ is invertible. 
Since $(X,\Delta, \ba^\lambda)$ is not log canonical, we can
pick a prime divisor $E$ on $Y$ such that $\ord_E(\Delta_Y+\lambda F) > 1$, where $\Delta_Y : = f^*(K_X+\Delta ) -K_Y$. 
Set $\varepsilon_0:=1$, $\varepsilon_n:= \min\{1, (\ord_E(\Delta_Y+ \lambda F) -1)/n \}$ for $n \ge 1$ and $\bb_n : = f_*\sO_Y(-nE)$ for $n \ge 0$. 

We will verify that the above $\{\varepsilon_n\}_{n \ge 0}$ and $\{\bb_n\}_{n \ge 0}$ satisfy the properties (i), (ii). 
Since $\bigcap_{n \ge 0}\bb_n=(0)$, the property (ii) follows from Chevalley's Theorem (see for example \cite[Exercise 8.7]{Mat}). 
For (i), we first observe that 
\[
\mathcal{J}(X, \Delta, \ba^\lambda \bb_n^{1-\varepsilon_n}) \subseteq f_* \sO_Y(-\lfloor \Delta_Y + \lambda F + (1-\varepsilon_n) nE \rfloor),
\]
because $\bb_n \sO_Y \subseteq \sO_Y(-nE)$.
Take any nonzero element $r \in \mathcal{J}(X,\Delta, \ba^\lambda \bb_n^{1-\varepsilon_n})$.
Then  
\begin{align*}
\ord_E \, r &\ge \ord_E(\lfloor \Delta_Y + \lambda F + (1-\varepsilon_n) nE \rfloor) \\
&= \lfloor \ord_E(\Delta_Y+\lambda F)+n-\varepsilon_nn \rfloor \\
&\ge n+1.
\end{align*} 
Combining this with the inclusion $r \in R$, we have $r \in \bb_{n+1}$ as desired.
\end{proof}

\begin{lem}\label{SFP and tau}
Let $(R,\m)$ be an $F$-finite normal local ring of characteristic $p>0$ and $\Delta$ be an effective $\Q$-Weil divisor on $X:=\Spec R$. 
Let $\ba \subseteq R$ be a nonzero ideal and $\lambda > 0$ be a real number.
If $(X,\Delta, \ba^\lambda)$ is sharply $F$-pure, then there exists a nonzero ideal $J \subseteq R$ such that
\[
J \subseteq \tau(X, \Delta, \ba^\lambda J^{1-\varepsilon})
\]
for every real number $0<\varepsilon \le 1$.
\end{lem}

\begin{proof}
Fix any real number $0<\varepsilon \le 1$. 
Since $(X,\Delta, \ba^\lambda)$ is sharply $F$-pure, there exist an  integer $e \ge 1$, a nonzero element $d \in \ba^{\lceil \lambda (p^e-1)\rceil}$ and an $R$-module homomorphism $\psi:F^e_*R(\lceil (p^e-1)\Delta \rceil) \to R$ sending $F^e_*d$ to $1$. 
By Remark \ref{sharp F-pure remark}, we may assume $e$ is large enough so that $\lceil (1-\varepsilon)p^e \rceil \le p^e-1$. 

Let $c \in R$ be a big sharp test element for the triple $(X, \Delta, \ba^{\lambda})$ (see Remark \ref{test element remark}). 
We take $J$ to be the principal ideal of $R$ generated by $c$. 
Then 
\[c=\psi(F^e_*(c^{p^e}d)) \in \psi(F^e_*(c  \ba^{\lceil \lambda (p^e-1)\rceil}J^{\lceil (1-\varepsilon)p^e\rceil})) \subseteq \tau(X, \Delta, \ba^\lambda J^{1-\varepsilon}),\]
where the last containment follows from Remark \ref{test element remark}. 
Thus, we have $J=cR \subseteq \tau(X, \Delta, \ba^\lambda J^{1-\varepsilon})$. 
\end{proof}

The following is the main result of this section. 
\begin{thm}\label{mixed lc}
Let $(R, \m)$ be a complete normal local ring of mixed characteristic $(0,p)$ and $\Delta$ be an effective $\Q$-Weil divisor on $X:=\Spec R$ such that $K_X+\Delta$ is $\Q$-Cartier. 
Suppose that $h_1, \dots, h_r$ is a regular sequence in $R$ such that $S:=R/(h_1, \dots, h_r)$ is an $F$-finite normal local ring of characteristic $p$ and $Z:=\Spec S$ is not contained in the support of $\Delta$.
Let $\lambda > 0$ be a real number and $\ba \subseteq R$ be an ideal not contained in the ideal $(h_1, \dots, h_r)$. 
If $(Z, \Delta|_Z, (\ba S)^{\lambda})$ is sharply $F$-pure, then $(X,\Delta, \ba^{\lambda})$ is log canonical.
\end{thm}

\begin{proof}
Assume to the contrary that $(X, \Delta, \ba^{\lambda})$ is not log canonical.
Let $\{\bb_n\}_{n \ge 0}$ and $\{\varepsilon_n\}_{n \ge 0}$ be as in Lemma \ref{non-lc}. 
By Lemma \ref{SFP and tau}, there exists a nonzero ideal $J \subseteq R$ such that $J \subseteq \tau(Z, \Delta|_Z, (\ba S)^\lambda J^{1-\varepsilon})$ for all $0 < \varepsilon \le 1$. 

We will show by induction that $J \subseteq \bb_n S$ for every integer $n \ge 0$.
The $n=0$ case is trivial because $\bb_0=R$. 
Therefore, suppose that the inclusion holds for some $n \ge 0$. 
Then
\begin{align*}
J \subseteq \tau(Z, \Delta|_Z, (\ba S)^\lambda J^{1-\varepsilon_n}) 
&\subseteq  \tau(Z, \Delta|_Z, (\ba S) ^\lambda (\bb_n S)^{1-\varepsilon_n})\\
&\subseteq  \tau_{\mathcal{B}}(X, \Delta, \ba^\lambda \bb_n^{1-\varepsilon_n}) S\\
&\subseteq  \mathcal{J}(X, \Delta, \ba^\lambda\bb_n^{1-\varepsilon_n})S \\
&\subseteq  \bb_{n+1} S,
\end{align*}
where the third containment follows from Lemma \ref{restriction}, the fourth one does from Lemma \ref{tauB and birat} and the fifth one does from Lemma \ref{non-lc} (i). 
Thus, we have $J \subseteq \bb_n S$ for every $n \ge 0$.

Combining the above inclusion with Lemma \ref{non-lc} (ii), we see that 
\[
J \subseteq \bigcap_{n \ge 0} \bb_n S \subseteq  \bigcap_{\ell \ge 1} \m_S^\ell=(0),
\]
which contradicts the fact that $J$ is a nonzero ideal.
\end{proof}

An inversion of adjunction type result also follows from a similar argument to the above one.  
We would like to thank Linquan Ma for pointing it out. 

\begin{thm}
Let $(R, \m)$ be a complete normal local ring of mixed characteristic $(0,p)$ and $\Delta$ be an effective $\Q$-Weil divisor on $X:=\Spec R$ such that $K_X+\Delta$ is $\Q$-Cartier. 
Suppose $Z$ is a $\Q$-Cartier prime divisor on $X$ which is $F$-finite of characteristic $p$ and not contained in the support of $\Delta$.
Let $\ba \subseteq \sO_X$ be an ideal not contained in $\sO_X(-Z)$ and $\lambda > 0$ be a real number. 
$Z^{N}$ denotes the normalization of $Z$ and $\mathrm{diff}_{Z^N}(Z+\Delta)$ denotes the different of $Z+\Delta$ on $Z^N$ $($see \cite[Subsection 2.1]{MSTWW} for the definition$)$.
If $(Z^{N}, \mathrm{diff}_{Z^N}(\Delta), (\ba \sO_{Z^{N}})^{\lambda})$ is sharply $F$-pure, then $(X,\Delta +Z, \ba^{\lambda})$ is log canonical.
\end{thm}

\begin{proof}
First note that for any real numbers $\delta, \mu>0$ and any ideal $\bb \subseteq \sO_X$, we have the inclusion
\[
\tau(Z^{N}, \mathrm{diff}_{Z^N}(Z+\Delta), (\ba \sO_{Z^N})^{\lambda} (\bb \sO_{Z^N})^\mu) \subseteq \tau_{\mathcal{B}}(X, \Delta + (1-\delta) Z, \ba^\lambda \bb^\mu) \sO_{Z^N}.
\]
Indeed, by Lemma \ref{test module for ideals}, it is enough to show that for any effective $\Q$-Cartier divisor $E$ on $X$ having no common component with $Z$, we have 
\[
\tau(Z^{N}, \mathrm{diff}_{Z^N}(Z+\Delta+E)) \subseteq \tau_{\mathcal{B}}(X, \Delta + (1-\delta) Z +E) \sO_{Z^N}.
\]
Take a big Cohen-Macaulay $R^+$-algebra $B$, with $R^+$-algebra homomorphism $B \to C$ to a big Cohen-Macaulay $S^+$-algebra $C$,  
such that 
\[
\tau_{\mathcal{B}}(X, \Delta + (1-\delta) Z +E) =\tau_B(X, \Delta + (1-\delta) Z +E). 
\]
The assertion then follows from \cite[Definition-Proposition 2.7]{MS} and \cite[Proposition 2.10 and Theorem 3.1]{MSTWW}. 

Now we assume to the contrary that $(X, \Delta +Z, \ba^{\lambda})$ is not log canonical.
Then there exists a rational number $\delta>0$ such that $(X, \Delta + (1-\delta)Z, \ba^{\lambda})$ is not log canonical.
We apply Lemma \ref{non-lc} to the triple $(X, \Delta + (1-\delta) Z, \ba^{\lambda})$ to obtain a descending chain $\{ \bb_n\}_{n \ge 0}$ of nonzero ideals of $\sO_X$ and a descending sequence $\{\epsilon_n\}_{n \ge 0}$ of positive real numbers. 
We also apply Lemma \ref{SFP and tau} to $(Z^{N}, \mathrm{diff}_{Z^N}(\Delta), (\ba \sO_{Z^{N}})^{\lambda})$ to obtain a nonzero ideal $J \subseteq \sO_{Z^N}$ such that 
\[
J \subseteq \tau(Z^N, \mathrm{diff}_{Z^N}(Z+\Delta), (\ba \sO_{Z^N})^\lambda J^{1-\varepsilon})
\] for all $0 < \varepsilon \le 1$. 
It then follows from an argument similar to the proof of Theorem \ref{mixed lc} that $J \subseteq \bb_n \sO_{Z^N}$ for all integers $n \ge 0$, which contradicts the fact that $J$ is a nonzero ideal. 
\end{proof}

\begin{cor}\label{mixed lc local}
With notation as in Setting \ref{local setting}, let $x \in X$ be a closed point and $\mathcal{Z} \subseteq \mathcal{X}$ be an irreducible closed subset which dominates $T$ and contains $x$.
Let $y$ be the generic point of $\mathcal{Z}$, which lies in the generic fiber $\mathcal{X}_\eta$. 
We further assume that the following three conditions hold: 
\begin{enumerate}[label=\textup{(\roman*)}]
\item $K_{\mathcal{X}}+\mathcal{D}$ is $\Q$-Cartier at $x$,  
\item the generic fiber $\mathcal{X}_\eta$ has characteristic zero, 
\item $X$ is defined over an $F$-finite field $k$ of characteristic $p>0$ and the triple $(X, \mathcal{D}|_X, (\ba \sO_X)^\lambda)$ is sharply $F$-pure at $x$.
\end{enumerate}
Then the triple $(\mathcal{X}_\eta, \mathcal{D}_\eta, \ba_\eta^\lambda)$ is log canonical at y.
\end{cor}

\begin{proof}
$R :=\widehat{\sO_{\mathcal{X}, x}}$ denotes the completion of the stalk $\sO_{\mathcal{X},x}$ and $S : =\widehat{\sO_{X, x}}$ denotes the completion of $\sO_{X,x}$.
Since $T$ is regular and $\mathcal{X}$ is flat over $T$, the kernel of the surjection $\sO_{\mathcal{X},x} \to \sO_{X,x}$ is generated by a regular sequence, and therefore, the same holds for the surjection $R \to S$.
Note that sharp $F$-purity is preserved under completion. 
It then follows from Theorem \ref{mixed lc} that $(R, f^*\Delta, (\ba R)^\lambda)$ is log canonical, where $f^*\Delta$ is the flat pullback of $\Delta$ under the canonical morphism $f: \Spec R \to \mathcal{X}$.
By Lemma \ref{regular map and lc}, the triple $(\mathcal{X}, \mathcal{D}, \ba^\lambda)$ is log canonical at $x$.
Taking into account that $y$ is a generalization of $x$, we conclude that $(\mathcal{X}, \mathcal{D}, \ba^\lambda)$ is log canonical at $y$.
\end{proof}

\begin{rem}
Using the theory of jet schemes, Zhu \cite[Corollary 4.2]{Zhu} proved Corollary \ref{mixed lc local} when the total space $\mathcal{X}$ is the affine space $\mathbb{A}^n_\Z$ over $\Z$. 
\end{rem}

\begin{rem}\label{algorithm}
Corollary \ref{mixed lc local} implies that one can use a method analogous to \cite[Algorithm 8.1]{MS} utilizing \cite{BHKK} to verify that a ring of finite type over $\Q$ has log canonical singularities.
\end{rem}

\begin{cor}\label{mixed lc proper}
With notation as in Setting \ref{local setting}, assume that the following four conditions hold: 
\begin{enumerate}[label=\textup{(\roman*)}]
\item $K_{\mathcal{X}}+\mathcal{D}$ is $\Q$-Cartier,  
\item the generic fiber $\mathcal{X}_\eta$ has characteristic zero,  
\item $X$ is defined over an $F$-finite field $k$ of characteristic $p>0$ and the triple $(X, \mathcal{D}|_X, (\ba \sO_X)^\lambda)$ is sharply $F$-pure,  
\item $\mathcal{X}$ is proper over $T$. 
\end{enumerate}
Then $(\mathcal{X}, \mathcal{D}, \ba^\lambda)$ is log canonical near $\mathcal{X}_\eta$, and in particular, the triple $(\mathcal{X}_\eta, \mathcal{D}_\eta, \ba_\eta^\lambda)$ is log canonical.
\end{cor}

\begin{proof}
Take any point $y \in \mathcal{X}_\eta$.
Since the structure map $\mathcal{X} \to T$ is a closed map, there exists a point $x \in X$ which is a specialization of $y$.
It then follows from Corollary \ref{mixed lc local} that $(\mathcal{X}, \mathcal{D}, \ba^\lambda)$ is log canonical at $y$.
\end{proof}

\section{Deformations with non-\texorpdfstring{$\Q$}{Q}-Gorenstein total space}

In this section, we study deformations of strongly $F$-regular/klt singularities when the total space is not necessarily $\Q$-Gorenstein. 

Throughout this section, we say that $(R, \Delta, \ba^{\lambda})$ is a \textit{triple} if $(R,\m)$ is an excellent normal local ring with dualizing complex, $\Delta$ is an effective $\Q$-Weil divisor on $\Spec R$, $\ba$ is a nonzero ideal of $R$, and $\lambda > 0$ is a real number. 

\begin{prop}\label{ideal twists}
Suppose that $(R, \Delta, \ba^{\lambda})$ is a triple and $A$ is an effective Weil divisor on $X:=\Spec R$ linearly equivalent to $- K_X$ such that  $B:=A-\Delta$ is also effective. 
Fix an integer $m \ge 1$ such that $m \Delta$ is an integral Weil divisor, and let $\bb \subseteq R$ be a nonzero ideal contained in $\sO_X(-mB)$. 
We further assume that one of the following three cases occurs. 
\begin{enumerate}[label=\textup{(\alph*)}]
\item $(R,\m)$ is a complete local domain of mixed characteristic $(0,p)$.
In this case, we set 
\[
I:= \tau_{\mathcal{B}}(X, A, \ba^\lambda \bb^{1-1/m}) \subseteq R.
\]
\item $(R,\m)$ is $F$-finite local ring of characteristic $p>0$. 
In this case, we set 
\[
I:= \tau(X, A, \ba^\lambda \bb^{1-1/m}) \subseteq R.
\]
\item $(R,\m)$ is a local ring of equal characteristic zero.  
In this case, we set 
\[
I:= \mathcal{J}(X, A, \ba^\lambda \bb^{1-1/m}) \subseteq R.
\]
\end{enumerate}
Then the following hold.
\begin{enumerate}
\item[\textup{(1)}] The ideal $I$ is contained in $\sO_X(-mB)$. 
\item[\textup{(2)}] Let $U \subseteq X$ 
be 
the locus where $m(K_X+\Delta)$ is Cartier. 
If $I|_U = \sO_X(-mB)|_U$, 
then $(U,\Delta|_U, \ba|_U^\lambda)$ is klt in the case $\textup{(a)}$ or $\textup{(c)}$ and is strongly $F$-regular in the case $\textup{(b)}$. 
\item[\textup{(3)}] Assume that $m (K_X+\Delta)$ is Cartier.
If $(X,\Delta, \ba^\lambda)$ is BCM-regular $($resp.~strongly $F$-regular, klt$)$ in the case $\textup{(a)}$ $($resp.~$\textup{(b)}$, $\textup{(c)}$$)$, then $\bb$ is contained in $I$. 
\end{enumerate}
\end{prop}

\begin{proof}
First we consider the case $\textup{(a)}$. 
For (1), since $\sO_X(-mB)$ is reflexive and $U$ is an open subset of $X$ whose complement has codimension at least two, it is enough to show that $I|_U \subseteq \sO_U(-m B|_U)$.
Noting that $\sO_U(-m B|_U)$ is invertible, we have
\begin{align*}
I|_U &\subseteq \mathcal{J}(U, A|_U, \ba|_U^\lambda  \bb|_U^{1-1/m})\\
&\subseteq \mathcal{J}(U, A|_U, \ba|_U^\lambda \sO_U(-mB|_U)^{1-1/m})\\
&= \mathcal{J}(U, A|_U +\frac{m-1}{m} (mB|_U), \ba|_U^\lambda)\\
&= \mathcal{J}(U, \Delta|_U + m B|_U, \ba|_U^\lambda)\\
&= \mathcal{J}(U, \Delta|_U, \ba|_U ^\lambda) \otimes_{\sO_U} \sO_U(-m B|_U),
\end{align*}
where the first inclusion follows from Lemma \ref{tauB and birat} and the last equality  follows from essentially the same argument as the proof of \cite[Proposition 9.2.31]{Laz}.
Therefore, we have $I|_U \subseteq \sO_V(-m B|_U)$ as desired.

If the equality $I|_U = \sO_X(-mB)|_U$ holds, then we see from the above inclusions that $\mathcal{J}(U, \Delta|_U, \ba|_U^\lambda) = \sO_U$, which proves (2).

For (3), we set $\q : = \bb  \sO_X(mB) \subseteq R$.
Then 
\begin{align*}
I &= \tau_{\mathcal{B}}(X, A, \ba^\lambda \q^{1-1/m} \sO_X(-mB)^{1-1/m})\\
&= \tau_{\mathcal{B}}(X, A + (m-1)B, \ba^\lambda  \q^{1-1/m})\\
&= \tau_{\mathcal{B}}(X, \Delta, \ba^\lambda \q^{1-1/m}) \sO_X(-mB),
\end{align*}
where the last equality follows from \cite[Lemma 6.6]{MS}.
On the other hand, by Definition \ref{BCM test ideal} and \cite[Lemma 6.6]{MS}, 
\begin{align*}
\tau_{\mathcal{B}}(X, \Delta, \ba^\lambda  \q^{1-1/m}) &\supseteq \tau_{\mathcal{B}}(X, \Delta+ \Div_X(f), \ba^\lambda)\\
&= f \tau_{\mathcal{B}}(X, \Delta, \ba^\lambda)\\
&= f R 
\end{align*}
for all nonzero elements $f \in \q$, 
where the last equality follows from the assumption that $(X,\Delta, \ba^{\lambda})$ is BCM-regular.
Therefore, $I \supseteq \q \sO_X(-mB) = \bb$, which completes the proof of (3).

In the case (b) (resp.~(c)), the assertion follows from a similar argument by replacing \cite[Lemma 6.6]{MS} with \cite[p.402 Basic Properties (ii)]{Tak0} (resp.~\cite[Proposition 9.2.31]{Laz}).
\end{proof}

The following theorem, the main result of this section, should be compared with Theorem \ref{mixed lc}. 
\begin{thm}\label{mixed klt}
Suppose that $(R,\Delta, \ba^\lambda)$ is a triple and $h$ is a nonzero element in $R$ such that $S:= R/(h)$ is normal. 
We assume in addition that $Z:=\Spec S$ is not contained in the support of $\Delta$, $K_Z+\Delta|_Z$ is $\Q$-Cartier, and $\ba$ is not contained in the ideal $(h)$. 
Let $U \subseteq X$ be  
the locus where $K_X+\Delta$ is $\Q$-Cartier.  
\begin{enumerate}[label=\textup{(\arabic*)}]
\item 
Suppose that $(R,\m)$ is a complete local domain of mixed characteristic $(0, p)$ and $S$ is an $F$-finite local domain of characteristic $p>0$. 
If $(Z, \Delta|_Z, (\ba S)^\lambda)$ is strongly $F$-regular, then $(U,\Delta|_U, \ba|_U^\lambda)$ is klt.
\item
Suppose that $(R,\m)$ is an $F$-finite local domain of characteristic $p>0$. 
If $(Z, \Delta|_Z, (\ba S)^\lambda)$ is strongly $F$-regular, then so is $(U,\Delta|_U, \ba|_U^\lambda)$. 
\item $(cf.~$\cite{EV}, \cite{Chi}$)$ 
Suppose that $(R,\m)$ is a local ring of equal characteristic zero. 
If $(Z, \Delta|_Z, (\ba S)^\lambda)$ is klt, then so is $(U,\Delta|_U, \ba|_U^\lambda)$. 
\end{enumerate}
\end{thm}

\begin{proof}
(1) Since $X$ is affine and Gorenstein at the generic point of $Z$, we can take an effective Weil divisor $A$ on $X$ linearly equivalent to $-K_X$ such that $B:= A -\Delta$ is effective and $\Supp A$ does not contain $Z$.
Take an integer $m \ge 1$ such that $m (K_Z + \Delta|_Z)$ is Cartier.
Set
\begin{align*}
I &: = \tau_{\mathcal{B}}(X, A, \ba^\lambda \bb^{1-1/m}) \subseteq R, \\ 
J &:= \tau(Z, A|_Z, (\ba S)^\lambda (\bb S)^{1-1/m}) \subseteq S,
\end{align*}
where we write  $\bb : = \sO_X(-mB) \subseteq R$.

Since $A|_Z$ is linearly equivalent to $-K_Z$, $B|_Z=A|_Z - \Delta|_Z$ and $\bb S \subseteq \sO_Z(-m B|_Z)$, 
we apply Proposition \ref{ideal twists} (3) with $X=Z$ and Lemma \ref{restriction} to deduce that $\bb S \subseteq J \subseteq IS$.
It follows from a combination of the inclusion $\bb S \subseteq IS$ with Proposition \ref{ideal twists} (1) that 
\[
I \subseteq \bb \subseteq I + \bb \cap (h).
\]
By assumption, $\Div_X(h)=Z$ is a prime divisor on $X$, which is not an irreducible component of $B$. Thus, $\bb \cap (h)=h(\bb:_R(h))=h \bb \subseteq \m \bb$, so that $\bb=I+\m \bb$. 
By Nakayama's lemma, we have $I=\bb$, which implies the assertion by Proposition \ref{ideal twists} (2). 

For (2), we set
\begin{align*}
I &: = \tau(X, A, \ba^\lambda \bb^{1-1/m}) \subseteq R,  \\ 
J &:= \tau(Z, A|_Z, (\ba S)^\lambda (\bb S)^{1-1/m}) \subseteq S.
\end{align*}
The proof then follows from an argument similar to the proof of (1), by replacing Lemma \ref{restriction} with \cite[Theorem 6.10 (1)]{HY}.

For (3), we set
\begin{align*}
I &: = \mathcal{J}(X, A, \ba^\lambda \bb^{1-1/m}) \subseteq R  \textup{ and}\\ 
J &:= \mathcal{J}(Z, A|_Z, (\ba S)^\lambda (\bb S)^{1-1/m}) \subseteq S.
\end{align*}
The proof then follows from an argument similar to the proof of (1) by replacing Lemma \ref{restriction} with \cite[Theorem 9.5.13]{Laz}.\footnote{\cite[Theorem 9.5.13]{Laz} 
is formulated for varieties, but the same statement for excellent $\mathbb{Q}$-schemes is obtained by using \cite[Theorem A]{Mur} instead of the local vanishing theorem.}
\end{proof}

\begin{cor}\label{mixed klt local}
With notation as in Setting \ref{local setting}, let $x \in X$ be a closed point and $\mathcal{Z} \subseteq \mathcal{X}$ be an irreducible closed subset which dominates $T$ and contains $x$.
Let $y$ be the generic point of $\mathcal{Z}$, which lies in the generic fiber $\mathcal{X}_\eta$.  
We further assume that the following conditions are all satisfied.  
\begin{enumerate}[label=\textup{(\roman*)}]
\item $T$ is a Dedekind scheme, that is, $\dim T=1$. 
\item One of the following holds.  
\begin{enumerate}[label=\textup{(\alph*)}]
\item $K_{X}+\mathcal{D}|_X$ is $\Q$-Cartier at $x$ and $K_{\mathcal{X}_{\eta}}+\mathcal{D}_{\eta}$ is $\Q$-Cartier at $y$, or 
\item $\dim \sO_{X,x} \le 2$. 
\end{enumerate}
\item One of the following cases occurs.
\begin{enumerate}[label=\textup{(\alph*)}]
\item $\sO_{T,t}$ is of mixed characteristic $(0,p)$, the residue field $\kappa(t)$ is $F$-finite of characteristic $p$, and $(X, \mathcal{D}|_X, (\ba \sO_X)^\lambda)$ is strongly $F$-regular at $x$, 
\item $\sO_{T,t}$ is $F$-finite of characteristic $p>0$ and $(X, \mathcal{D}|_X, (\ba \sO_X)^\lambda)$ is strongly $F$-regular at $x$, or
\item $\sO_{T,t}$ is of equal characteristic zero and $(X, \mathcal{D}|_X, (\ba \sO_X)^\lambda)$ is klt at $x$.
\end{enumerate}
\end{enumerate}
Then $(\mathcal{X}_\eta, \mathcal{D}_\eta, \ba_\eta^\lambda)$ is klt at $y$ in the case $\textup{(iii-a)}$ or $\textup{(iii-c)}$ and is strongly $F$-regular at $y$ in the case $\textup{(iii-b)}$. 
\end{cor}

\begin{proof}
First, we assume that (ii-a) holds. 
For the case (iii-a), $R:=\widehat{\sO_{\mathcal{X}, x}}$ denotes the completion of the stalk $\sO_{\mathcal{X},x}$ and $S:=\widehat{\sO_{X, x}}$ denotes the completion of $\sO_{X,x}$.
Note that $S$ is $F$-finite, because $\sO_{X,x}$ is essentially of finite type over an $F$-finite field $\kappa(t)$. 
Since the morphism $\Spec R \to \Spec \sO_{\mathcal{X}, x}$ induced by the completion $\sO_{\mathcal{X},x} \to R$ is surjective and $y$ is a generalization of $x$, there exists a point $y' \in \Spec R$ such that $f(y')=y$, where $f: \Spec R \to \mathcal{X}$ is the canonical morphism. 
$f^*\Delta$ denotes the flat pullback of $\Delta$ by $f$, and then $K_{\Spec R} + f^* \Delta$ is $\Q$-Cartier at $y'$ by Lemma \ref{regular map and lc}. 
We see from the fact that $T$ is a Dedekind scheme and $\mathcal{X}$ is flat over $T$ that the kernel of the surjection $R \to S$ is a principal ideal. 
Since strong $F$-regularity is preserved under completion, we apply Theorem \ref{mixed klt} to deduce that $(\Spec R, f^*\Delta, (\ba R)^\lambda)$ is klt at $y'$.
Therefore, by Lemma \ref{regular map and lc} again, $(\mathcal{X}, \mathcal{D}, \ba^\lambda)$ is klt at $y$. 
The cases (iii-b) and (iii-c) follow similarly, by replacing $R$ with $\sO_{\mathcal{X},x}$ and $S$ with $\sO_{X,x}$.

Next, we assume that (ii-b) holds. 
It suffices to show that the condition (iii) implies the condition (ii-a). 
First we consider the case (iii-a). 
The closed fiber $X$ is $F$-rational at $x$ by Remark \ref{SFR to F-rat}.  
Therefore, the log $\Q$-Gorensteinness of $(X,\mathcal{D}|_X)$ is an immediate consequence of Remark \ref{remark rat} (ii). 
Also, it follows from \cite[Theorem 3.8]{MS} that $\mathcal{X}$ is pseudo-rational at $x$ and, in particular, at $y$, because $y$ is a generalization of $x$. 
Thus, $\mathcal{X}_{\eta}$ is pseudo-rational at $y$, and by Remark \ref{remark rat} (ii) again, $K_{\mathcal{X}_{\eta}}+\mathcal{D}_{\eta}$ is $\Q$-Cartier at $y$. 
The case (iii-b) follows similarly, by replacing \cite[Theorem 3.8]{MS} with \cite[Theorem 4.2 (h)]{HH2} and Remark \ref{SFR to F-rat}. 
The case (iii-c) does, by replacing Remark \ref{SFR to F-rat} with Remark \ref{remark rat} (iii) and \cite[Theorem 3.8]{MS} with \cite{Elk}.
\end{proof}

In Corollary \ref{mixed klt local}, the log $\Q$-Gorenstein assumption on the generic fiber is essential. 

\begin{eg}[$\textup{\cite[Theorem 1.1]{Sin}, cf. \cite[Remark 6.5]{DSS}}$]\label{counter eg local}
  We give an example of 
  \[(\mathcal{X}, \mathcal{D}, \ba^\lambda, \mathcal{Z}, T, x, y, t)\] 
  which satisfies all the assumptions, except for the condition that $K_{\mathcal{X}_{\eta}}+\mathcal{D}_{\eta}$ is $\Q$-Cartier at $y$, of Corollary \ref{mixed klt local} and where $(\mathcal{X}_{\eta}, \mathcal{D}_{\eta}, \ba^{\lambda}_{\eta})$ is not klt at $y$ in the sense of de Fernex-Hacon (cf.~\cite[Section 7]{dFH}). 
  The latter condition means that there does not exist an effective $\Q$-Weil divisor $\Delta$ on $\mathcal{X}_{\eta}$ such that $K_{\mathcal{X}_{\eta}} + \mathcal{D}_{\eta}+\Delta$ is $\Q$-Cartier at $y$ and $(\mathcal{X}_{\eta}, \mathcal{D}_{\eta} + \Delta, \ba^{\lambda}_{\eta})$ is klt at $y$.
  
  Fix integers $m, n \ge 1 $ such that $m-m/n >2$.
  Let $I$ be the ideal of a polynomial ring $\Z[A,B,C,D,E]$ generated by the size two minors of the matrix
  \[
  \left(
  \begin{array}{ccc}
  A^2+(3E)^m & B & D \\
  C & A^2 & B^n -D
  \end{array}
  \right)
  \]
  and we write $R: =\Z[A,B,C, D, E]/I$.
  We set 
  \[
    \mathcal{X}:= \Spec R, \; 
    \mathcal{D}:= 0, \;
    \ba:= \sO_{\mathcal{X}}, \;
    \lambda:=1, \;
    T:= \Spec \Z, 
  \]
and let $f:\mathcal{X} \to T$ be a canonical morphism and  $\mathcal{Z}$ be the closed subscheme of $\mathcal{X}$ defined by the ideal $(A,B,C,D,E) \subseteq R$. 
Then all the fibers of $f|_{\mathcal{Z}}:\mathcal{Z} \to T$ are singletons. 
We choose $t \in T$ as the closed point corresponding to the prime number 3, and let $x \in \mathcal{X}$ be 
the unique element of the fiber $\mathcal{Z}_t$ over $t$ and $y \in \mathcal{X}$ be  
the unique element of the generic fiber $\mathcal{Z}_\eta$. 
  
First we note that $X := \mathcal{X}_t = \Spec R/(3) = \Spec S \times_{\F_3} \mathbb{A}^1_{\F_3}$, where $S$ is the quotient ring of a polynomial ring $\F_3[A, B, C, D]$ by the ideal generated by the size two minors of the matrix
  \[
  \left(
  \begin{array}{ccc}
  A^2 & B & D \\
  C & A^2 & B^n -D
  \end{array}
  \right).
  \] 
Since $S$ is strongly $F$-regular and $\Q$-Gorenstein by the proof of \cite[Proposition 4.3]{Sin}, so is $X$.
  
Next we show that $\mathcal{X}_{\eta}$ is not klt at $y$ in the sense of de Fernex-Hacon.
Assume to the contrary that there exists an effective $\Q$-Weil divisor $\Delta$ on $\mathcal{X}_{\eta}$ such that $K_{\mathcal{X}_{\eta}} + \Delta$ is $\Q$-Cartier and $(\mathcal{X}_{\eta},  \Delta)$ is klt at $y$. 
It then follows from \cite[Corollary 3.4]{Tak0} that for general prime numbers $p$, the fiber $\mathcal{X}_p$ of $f$ over the closed point $(p) \in T$ is strongly $F$-regular at $x_p$, where $x_p \in \mathcal{X}$ is the unique element of the fiber $\mathcal{Z}_p$ over $(p)$.  
However, by \cite[Theorem 1.1]{Sin}, $\mathcal{X}_p$ is not strongly $F$-regular at $x_p$ if $p>3$, which is a contradiction.
Therefore, $\mathcal{X}_{\eta}$ is not klt in the sense of de Fernex-Hacon, as desired.
\end{eg}

In order to discuss singularities of geometric generic fibers, we introduce the notion of geometrically strongly $F$-regular triples. 
\begin{defn}\label{defn GSFR}
Suppose that $X$ is a scheme essentially of finite type over an $F$-finite field $k$ of characteristic $p>0$ and $\Delta$ is an effective $\Q$-Weil divisor on $X$. 
Let $\ba \subseteq \sO_X$ be a nonzero coherent ideal sheaf and $\lambda>0$ be a real number. 
$X_l:= X \times_{\Spec k} \Spec l$ 
denotes the base change of $X$ to the perfect closure $l:= k^{1/p^{\infty}}$ of $k$ and 
$\Delta_l$ (resp. $\ba_l$) denotes the flat pullback of $\Delta$ (resp. $\ba$) to $X_l$. 
  
\begin{enumerate}[label=(\roman*)]
\item We say that $(X, \Delta, \ba^\lambda)$ is \textit{geometrically strongly $F$-regular} over $k$ at a point $x \in X$ if $(X_l, \Delta_l, \ba_l^\lambda)$ is strongly $F$-regular at the (unique) point $x_l \in X_l$ lying over $x$.
  
\item We say that $(X, \Delta, \ba^\lambda)$ is \emph{geometrically strongly $F$-regular} over $k$ if it is geometrically strongly $F$-regular at every point $x \in X$.
\end{enumerate}
\end{defn}

Using the theory of relative test ideals (see Appendix \ref{relative test ideals}), we establish the following equivalent criteria for a triple to be geometrically strongly $F$-regular.  
\begin{prop}\label{GSFR}
  Suppose that $X$ is a geometrically normal scheme essentially of finite type over an $F$-finite field $k$ of characteristic $p>0$ and $\Delta$ is an effective $\Q$-Weil divisor on $X$ such that $K_X+\Delta$ is $\Q$-Cartier with index not divisible by $p$. 
  Let $\ba \subseteq \sO_X$ be a nonzero coherent ideal sheaf and  $\lambda$ be a rational number. 
  For a field extension $k \subseteq l$, $X_l$ denotes the fiber product $X \times_{\Spec k} \Spec l$ and $\Delta_l$ (resp.~$\ba_l$) denotes the flat pullback of $\Delta$ (resp.~$\ba$) to $X_l$. 
  Given a point $x \in X$, the following conditions are equivalent to each other.
  \begin{enumerate}[label=\textup{(\roman*)}]
  \item $(X, \Delta, \ba^\lambda)$ is geometrically strongly $F$-regular over $k$ at $x$.
  \item For any perfect field $l \supseteq k$, $(X_l, \Delta_l, \ba_l^\lambda)$ is strongly $F$-regular at every point of $X_l$ lying over $x$.
  \item For any $F$-finite field $l \supseteq k$, $(X_l, \Delta_l, \ba_l^\lambda)$ is strongly $F$-regular at every point of $X_l$ lying over $x$.
  \item For a sufficiently divisible integer $n \ge 1$, $(X_{k^{1/p^n}}, \Delta_{k^{1/p^n}}, \ba_{k^{1/p^n}}^\lambda)$ is strongly $F$-regular at a (unique) point $X_{k^{1/p^n}}$ lying over $x$.
  \end{enumerate}
  \end{prop}
  
  \begin{proof}
  The conditions (i), (ii), and (iv) are equivalent by Theorem \ref{absolute tau vs fiber}. 
  Since the implication (iii)$\Rightarrow$(iv) is obvious, we will show that (ii) implies (iii). 
  
  Given an $F$-finite field $l \supseteq k$, let $l'$ be the perfect closure of $l$. 
  By (ii), $(X_{l'}, \Delta_{l'}, \ba_{l'}^\lambda)$ is strongly $F$-regular at every point of $X_{l'}$ lying over $x$. 
  Applying an argument similar to the proof of \cite[Theorem 3.1]{HH} to the faithfully flat morphism $X_{l'} \to X_l$, we see that $(X_l, \Delta_l, \ba_l^\lambda)$ is strongly $F$-regular at every point of $X_l$ lying over $x$.
  \end{proof}

\begin{thm}\label{deformation GSFR}
With notation as in Setting \ref{local setting}, assume that the conditions $\textup{(i)}$, $\textup{(ii)}$, and $\textup{(iii-b)}$ in Corollary \ref{mixed klt local} hold. 
If the residue field $\kappa(t)$ is perfect, 
then $(\mathcal{X}_\eta, \mathcal{D}_\eta, \ba_\eta^\lambda)$ is geometrically strongly $F$-regular over the function field $\kappa(\eta)$ at $y$.
\end{thm}

\begin{proof}
First note that the fiber $X$ is geometrically normal over $\kappa(t)$ at $x$ because $\kappa(t)$ is perfect.
Shrinking $\mathcal{X}$ and $T$ if necessary, we may assume that $\mathcal{X}$ is affine and $g: \mathcal{X} \to T$ is a normal morphism, that is, its fibers are geometrically normal. 
Take an effective $\Q$-Weil divisor $A$ on $\mathcal{X}$, having no common component with $X$, linearly equivalent to $K_{\mathcal{X}} + \mathcal{D}$. 
Replacing $\mathcal{D}$ by $\mathcal{D}+ \epsilon A$ with sufficiently small $\epsilon>0$, we may assume that the index of $K_{\mathcal{X}_\eta}+\mathcal{D}_\eta$ is not divisible by $p$.
Similarly, we may assume that $\lambda$ is a rational number.

Fix an integer $n \ge 1$ and set $L : = \kappa(\eta)^{1/p^n}$.
Let $T'$ be the normalization of $T$ in $L$ and $\mathcal{X}' : = \mathcal{X} \times_T T'$ be the base change of $\mathcal{X}$ to $T'$. 
We note that $T'$ is a Dedekind scheme
and $\mathcal{X}'$ is normal.
Let $\eta'$ be the generic point of $T'$ and $y' \in \mathcal{X}'_{\eta'} \cong \mathcal{X}_\eta \times_{\Spec \kappa(\eta)}  \Spec L$ be the unique point lying over $y \in \mathcal{X}_\eta$. 
By Proposition \ref{GSFR}, it is enough to show that $(\mathcal{X}'_{\eta'}, \mathcal{D}'_{\eta'}, \ba'^\lambda_{\eta'})$ is strongly $F$-regular at $y'$, where $\ba':= \ba \sO_{\mathcal{X}'}$ and $\mathcal{D}' : = \mu^* \mathcal{D}$ is the pullback of $\mathcal{D}$ by the finite flat morphism $\mu : \mathcal{X}' \to \mathcal{X}$.

We verify that $K_{\mathcal{X}'_{\eta'}} + \mathcal{D}'_{\eta'}$ is $\Q$-Cartier at $y'$.
Let $\omega_{\mathcal{X}_{\eta}/\kappa(\eta)}$ and $\omega_{\mathcal{X}'_{\eta'}/L}$ be the relative canonical sheaves of $\mathcal{X}_{\eta} \to \Spec \kappa(\eta)$ and $\mathcal{X}'_{\eta'} \to \Spec L$, respectively.
Since $\Spec \kappa(\eta)$ and $\Spec L$ are Gorenstein, the relative canonical sheaves are nothing but canonical sheaves by \cite[Lemma 0BZL]{Sta}.
Moreover, by \cite[Lemma 0BZV]{Sta}, we have an isomorphism $\mu_{\eta}^* \omega_{\mathcal{X}_{\eta}/\kappa(\eta)} \cong \omega_{\mathcal{X}'_{\eta'}/L}$, where $\mu_{\eta}: \mathcal{X}'_{\eta'} \to \mathcal{X}_{\eta}$ is the induced morphism by $\mu$.
Taking into account that a canonical divisor is unique up to adding a Cartier divisor (\cite[V. Theorem 3.1]{Hart}), 
we conclude that $K_{\mathcal{X}'_{\eta'}} - \mu_{\eta}^* K_{\mathcal{X}_{\eta}}$ is a Cartier divisor.
Thus, $K_{\mathcal{X}'_{\eta'}} + \mathcal{D}'_{\eta'}$ is $\Q$-Cartier at $y'$. 

Finally, we show that $(\mathcal{X}'_{\eta'}, \mathcal{D}'_{\eta'}, \ba'^{\lambda}_{\eta'})$ is strongly $F$-regular at $y'$. 
Take a point $t' \in T'$ lying over $t \in T$ and a point $x' \in X' : =\mathcal{X}'_{t'}$ lying over $x \in X=\mathcal{X}_t$.
Since $\kappa(t)$ is perfect and $X' \cong X \times_{\Spec \kappa(t)} \Spec \kappa(t')$, it follows from Proposition \ref{GSFR} that $(X', \mathcal{D'}|_{X'}, (\ba' \sO_{X'}) ^\lambda)$ is strongly $F$-regular at $x'$.
Moreover, $K_{X'} + \mathcal{D'}|_{X'}$ is $\Q$-Cartier at $x'$ and $x' \in \mathcal{X}'$ is a specialization of $y'$.
It then follows from Corollary \ref{mixed klt local} that $(\mathcal{X}'_{\eta'}, \mathcal{D}'_{\eta'}, \ba'^{\lambda}_{\eta'})$ is strongly $F$-regular at $y'$ as desired.
\end{proof}

As a corollary of Corollary \ref{mixed klt local} and Theorem \ref{deformation GSFR}, we give an affirmative answer to a conjecture of Liedtke-Martin-Matsumoto \cite{LMM} on deformations of isolated lrq singularities. 
An \textit{isolated lrq singularity} over a field $k$ is the spectrum $\Spec R$ of a normal local $k$-algebra $(R,\m)$ such that 
\[
\widehat{R} \cong k[[x_1, \dots, x_d]]^G,
\]
where $\widehat{R}$ is the $\m$-adic completion of $R$ and $G$ is a finite linearly reductive group scheme, acting on a formal power series ring $k[[x_1, \dots, X_d]]$ over $k$, whose action fixes the closed point and is free away from it (see \cite[Definition 6.4]{LMM}). 

\begin{cor}[\textup{\cite[Conjecture 12.1 (1)]{LMM}}]
Let $B$ be the spectrum of a DVR with an algebraically closed residue field $k$ and $\mathcal{X} \to B$ be a flat morphism of finite type with special and geometric generic fibers $\mathcal{X}_0$ and $\mathcal{X}_{\overline{\eta}}$, respectively. 
Let $x \in \mathcal{X}_0$ and $y \in \mathcal{X}_{\overline{\eta}}$ be points such that $x \in \mathcal{X}$ is a specialization of the image $u(y) \in \mathcal{X}$ of $y$ by the morphism $u:\mathcal{X}_{\overline{\eta}} \to \mathcal{X}$.
If the special fiber $\mathcal{X}_0$ has an isolated lrq singularity at $x$, then so does the geometric generic fiber $\mathcal{X}_{\overline{\eta}}$ at $y$. 
\end{cor}

\begin{proof}
By \cite[Lemma 12.5]{LMM}, the problem is reduced to showing that if $B$ is of mixed characteristic $(0,p)$ (resp.  characteristic $p>0$) and $\mathcal{X}_0$ is two-dimensional strongly $F$-regular at $x$, then $\mathcal{X}_{\overline{\eta}}$ is klt (resp. strongly $F$-regular) at $y$.
The mixed characteristic case follows from the case (iii-a) of Corollary \ref{mixed klt local}, and therefore, we consider the case where $B$ is of characteristic $p$.  
Since strong $F$-regularity descends under faithfully flat morphisms by \cite[Theorem 3.1 b)]{HH}, passing to the completion, we may assume that $B$ is the spectrum of a complete DVR. 
Then $B$ is $F$-finite and the assertion follows from Theorem \ref{deformation GSFR}. 
\end{proof}

Next, we study the behavior of singularities in proper flat families. 

\begin{cor}\label{mixed klt proper}
With notation as in Setting \ref{local setting}, assume that the following conditions are all satisfied. 
\begin{enumerate}[label=\textup{(\roman*)}]
\item $T$ is a Dedekind scheme, that is, $\dim T=1$. 
\item One of the following holds. 
\begin{enumerate}[label=\textup{(\alph*)}]
\item 
$K_{X}+\mathcal{D}|_X$ and $K_{\mathcal{X}_{\eta}}+\mathcal{D}_{\eta}$ are $\Q$-Cartier, or 
\item
$\dim X  \le 2$.
\end{enumerate}
\item One of the following cases occurs.
\begin{enumerate}[label=\textup{(\alph*)}]
\item $\sO_{T,t}$ is of mixed characteristic $(0,p)$, the residue filed $\kappa(t)$ is $F$-finite of characteristic $p$, and $(X, \mathcal{D}|_X, (\ba \sO_X)^\lambda)$ is strongly $F$-regular,
\item $\sO_{T,t}$ is $F$-finite of characteristic $p>0$ and $(X, \mathcal{D}|_X, (\ba \sO_X)^\lambda)$ is strongly $F$-regular, or
\item $\sO_{T,t}$ is of equal characteristic zero and $(X, \mathcal{D}|_X, (\ba \sO_X)^\lambda)$ is klt.
\end{enumerate}
\item $\mathcal{X}$ is proper over $T$. 
\end{enumerate}
Then $(\mathcal{X}_\eta, \mathcal{D}_\eta, \ba_\eta^\lambda)$ is klt in the case $\textup{(iii-a)}$ or $\textup{(iii-c)}$ and is strongly $F$-regular in the case $\textup{(iii-b)}$. 
Moreover, in the case \textup{(iii-b)}, if the residue field $\kappa(t)$ is perfect, then $(\mathcal{X}_\eta, \mathcal{D}_\eta, \ba_\eta^\lambda)$ is geometrically strongly $F$-regular over the function field $\kappa(\eta)$. 
\end{cor}

\begin{proof}
Take any point $y \in \mathcal{X}_\eta$.
Since the structure map $\mathcal{X} \to T$ is a closed map, there exists a point $x \in X$ which is a specialization of $y$.
It then follows from Corollary \ref{mixed klt local} that $(\mathcal{X}, \mathcal{D}, \ba^\lambda)$ is klt (resp.~strongly $F$-regular) at $y$ in the case (a) or (c) (resp. the case (b)). 
The assertion on geometric strong $F$-regularity follows from Theorem \ref{deformation GSFR}. 
\end{proof}

\begin{cor}
With notation as in Setting \ref{local setting}, assume that the conditions $\textup{(i)}$, $\textup{(ii)}$, and $\textup{(iv)}$ in Corollary \ref{mixed klt proper} hold.  
If $T$ is of finite type over $\Spec \Z$, 
then the following conditions are equivalent to each other: 
\begin{enumerate}[label=\textup{(\roman*)}]
\item 
$(\mathcal{X}_{p}, \mathcal{D}_p, \ba_p^\lambda)$ is strongly $F$-regular for some closed point $p \in T$,
\item 
$(\mathcal{X}_p, \mathcal{D}_p, \ba_p^\lambda)$ is strongly $F$-regular for a general closed point $p \in T$, 
\item 
$(\mathcal{X}_\eta, \mathcal{D}_\eta, \ba_{\eta}^\lambda)$ is klt. 
\end{enumerate}
\end{cor}

\begin{proof}
It is obvious that (ii) implies (i). The implication (i)$\Rightarrow$(iii) is a consequence of Corollary \ref{mixed klt proper}. 
The implication (iii)$\Rightarrow$(ii) follows from the fact that modulo $p$ reduction of a klt singularity is strongly $F$-regular for general $p$ (see \cite{HY} and \cite{Tak0}). 
\end{proof}

\begin{cor}\label{general SFR fibers}
Let $T$ be a smooth curve over a perfect field $k$ of characteristic $p>0$ and $(\mathcal{X}, \mathcal{D}, \ba^{\lambda}) \to T$ be a proper flat family of triples over $T$, where $\mathcal{D}$ is an effective $\Q$-Weil divisor on a normal variety $\mathcal{X}$ over $k$, $\ba \subseteq \sO_X$ is a nonzero coherent ideal sheaf, and $\lambda > 0$ is a real number. 
\begin{enumerate}[label=\textup{(\arabic*)}]
\item 
Suppose that $k$ is an uncountable algebraically closed field. 
If some closed fiber $(\mathcal{X}_{t_0}, \mathcal{D}_{t_0}, (\ba \sO_{\mathcal{X}_{t_0}})^{\lambda})$ is log $\Q$-Gorenstein and strongly $F$-regular and if a general closed fiber $(\mathcal{X}_t, \mathcal{D}_t)$ is log $\Q$-Gorenstein, then the general fiber $(\mathcal{X}_{t}, \mathcal{D}_{t}, (\ba \sO_{\mathcal{X}_{t}})^{\lambda})$ is also strongly $F$-regular.
\item If some closed fiber $(\mathcal{X}_{t_0}, \mathcal{D}_{t_0}, (\ba \sO_{\mathcal{X}_{t_0}})^{\lambda})$ is two-dimensional strongly $F$-regular, then so is a general fiber $(\mathcal{X}_{t}, \mathcal{D}_{t}, (\ba \sO_{\mathcal{X}_{t}})^{\lambda})$. 
\end{enumerate}
\end{cor}

\begin{proof}
(1) First note that the generic fiber $(\mathcal{X}_{\eta}, \mathcal{D}_{\eta}, (\ba \sO_{\mathcal{X}_{\eta}})^{\lambda})$ is log $\Q$-Gorenstein by Remark \ref{Q-Gorenstein}. 
Since the closed fiber $\mathcal{X}_{t_0}$ is geometrically normal over $\kappa(t_0)=k$, shrinking $T$ if necessary, we may assume that all fibers of $\mathcal{X} \to T$ is geometrically normal.
On the other hand, it follows from Corollary \ref{mixed klt proper} that $(\mathcal{X}_{\eta}, \mathcal{D}_{\eta}, \ba_\eta^\lambda)$ is geometrically strongly $F$-regular over the function field $\kappa(\eta)$. 
Since $K_{\mathcal{X}} + \mathcal{D}$ is $\Q$-Cartier along $\mathcal{X}_{\eta}$, there exists a non-empty open subset $V \subseteq T$ such that $K_{\mathcal{X}}+\mathcal{D}$ is $\Q$-Cartier on $\mathcal{X}_V : = \mathcal{X} \times_T V$.
Applying Proposition \ref{fibrewise GSFR} to $\mathcal{X}_V$, we see that the fiber $(\mathcal{X}_s, \mathcal{D}_s, \ba_s^\lambda)$ over a general closed point $s \in V$ is geometrically strongly $F$-regular over $\kappa(s)$. 
In particular, $(\mathcal{X}_s, \mathcal{D}_s, \ba_s^\lambda)$ is strongly $F$-regular by Proposition \ref{GSFR}.

(2) The proof is essentially the same as that of (1), but the log $\Q$-Gorensteinness of the generic fiber $(\mathcal{X}_{\eta}, \mathcal{D}_{\eta})$ follows from Corollary \ref{mixed klt local}, and therefore, the extra assumption on the base field $k$ is unnecessary. 
\end{proof}

\begin{rem}
When $\mathcal{D}=0$, we do not need to assume the normality of $\mathcal{X}$ in Corollaries \ref{mixed klt local}, \ref{mixed klt proper}, and \ref{general SFR fibers}. 
In this case, since normality lifts from Cartier divisors, by shrinking $\mathcal{X}$ and $T$ in Corollary \ref{mixed klt local} (resp.~by shrinking $T$ in Corollaries \ref{mixed klt proper} and \ref{general SFR fibers}), we can reduce to the case where $\mathcal{X}$ is normal. 
\end{rem}

In the proof of Corollary \ref{general SFR fibers}, we used the following proposition, which shows that the strong $F$-regularity of general fibers is deduced from that of geometric generic fibers.  
\begin{prop}[$\textup{cf. \cite[Corollary 4.21]{PSZ}}$]\label{fibrewise GSFR}
  Suppose that $V$ is an $F$-finite regular integral scheme of characteristic $p>0$, $f: X \to V$ is a flat morphism of finite type from a normal integral scheme $X$, and $\Delta$ is an effective $\Q$-Weil divisor on $X$ such that $K_X+\Delta$ is $\Q$-Cartier. 
  Let $\ba$ be a nonzero coherent ideal sheaf and $\lambda$ be a real number, and $\eta \in V$ denotes the generic point.
  If the generic fiber $(X_{\eta}, \Delta_{\eta}, \ba_\eta^\lambda)$ is geometrically strongly $F$-regular over the function field $\kappa(\eta)$ of $V$, then $(X_s, \Delta_s, \ba_s^\lambda)$ is geometrically strongly $F$-regular over $\kappa(s)$ for a general point $s \in V$.
\end{prop}
  
  \begin{proof}
  Taking an affine open covering of $X$, we can assume that $X$ is affine, and then there exists an effective $\Q$-Weil divisor $D$ on $X$ such that $D \sim K_X+\Delta$.
  Since $(X_\eta, \Delta_\eta + \epsilon D_\eta, \ba_\eta^\lambda)$ is geometrically strongly $F$-regular over $\kappa(\eta)$ for any sufficiently small $\epsilon>0$,  replacing $\Delta$ by $\Delta + \epsilon D$, 
  we may assume that the index of $K_X+\Delta$ is not divisible by $p$.
  Similarly, we may assume that $\lambda$ is a rational number. 
  
  It follows from Theorem \ref{absolute tau vs fiber} that, possibly after shrinking $V$, there exists an integer $n$ such that $(X_s, \Delta_s, \ba_s^\lambda)$ is geometrically strongly $F$-regular over $\kappa(s)$ for every point $s \in V$ if and only if $\tau(X_{V^{n}}, h^* \Delta, (\ba \sO_{X_{V^{n}}}) ^\lambda)=\sO_{X_{V^{n}}}$ near $h^{-1}(X_s)$, where $h: X_{V^{n}}=X \times_V V^n \to X$ is the first projection. 
  Let $Z \subseteq X_{V^n}$ be the closed subscheme defined by the ideal $\tau(X_{V^{n}}, h^* \Delta, (\ba \sO_{X_{V^{n}}})^\lambda)$.
  By Chevalley's theorem on constructible sets, $(f\circ h)(Z) \subseteq V$ is a constructible set.
  Since the complement $V \setminus (f\circ h)(Z)$ is a constructible set containing the generic point $\eta$, it contains a dense open subset $U \subseteq V$.
  By the definition of $Z$, we see that $(X_s, \Delta_s, \ba_s^\lambda)$ is geometrically strongly $F$-regular over $\kappa(s)$ for every point $s \in U$. 
  \end{proof}

We close this section with an example showing the global analog of Corollary \ref{mixed klt proper} does not hold. 
\begin{defn}
Let $X$ be a normal projective variety over a perfect field $k$ and $\Delta$ be an effective $\Q$-Weil divisor on $X$. 
\begin{enumerate}
  \item The pair $(X,\Delta)$ is said to be \textit{log Fano} if $(X,\Delta)$ is klt and $-(K_X+\Delta)$ is ample. We say that $X$ is of \textit{Fano type} if there exists an effective Weil divisor $B$ on $X$ such that $(X,B)$ is a log Fano pair. 
  \item (\textup{\cite{HWY}, \cite{Smi2}}) Suppose that $k$ is of characteristic $p>0$. 
  Then $X$ is said to be \textit{globally $F$-regular} if for every effective Weil divisor $D$ on $X$, there exists an integer $e \ge 1$ such that the composite 
\[
\sO_X \xrightarrow{\varphi^{(e)}_0} F^e_*\sO_X \to F^e_*\sO_X(D)  
\]
splits as an $\sO_X$-module homomorphism, where $\varphi^{(e)}_0$ is defined as in Definition \ref{F-sing def} and $F^e_*\iota$ is the pushforward of the natural inclusion $\iota:\sO_X \to \sO_X(D)$ by the $e$-times iterated Frobenius morphism $F^e:X \to X$. 
\end{enumerate}
\end{defn}

\begin{prop}[\textup{cf.~\cite[Proposition 5.4]{SS}}]\label{Fano klt}
Let $(X,\Delta)$ be a log Fano pair over an algebraically closed field $k$ of characteristic zero and $H$ be an ample Cartier divisor on $X$.
Let $S=\bigoplus_{n \ge 0}H^0(X, \sO_X(nH))$ 
be the section ring of $X$ with respect to $H$ and $\Delta_S$  
be the $\Q$-Weil divisor on $\Spec S$ corresponding to $\Delta$. 
Then $(\Spec S, \Delta_S)$ is klt in the sense of de Fernex-Hacon. 
\end{prop}

\begin{proof}
Take a sufficiently small $\epsilon<<1$ so that $-(K_X+\Delta)-\epsilon H$ is ample. 
By Bertini, we can take $\Delta' \in | -(K_X+\Delta)-\epsilon H |_{\Q}$ so that $(X, \Delta+\Delta')$ is klt.
It then follows from an argument similar to the proof of \cite[Proposition 5.4]{SS} that $(\Spec S, \Delta_S+ \Delta'_S)$ is klt, where $\Delta_S'$ is the $\Q$-Weil divisor on $\Spec S$ corresponding to $\Delta'$. 
In particular, $(\Spec S, \Delta_S)$ is klt in the sense of de Fernex-Hacon. 
\end{proof}

Since globally $F$-regular varieties (resp.~varieties of Fano type) can be viewed as a global analog of strongly $F$-regular singularities (resp.~klt singularities) (see \cite[Proposition 5.3, 5.4]{SS}), 
it is natural to ask whether globally $F$-regular varieties deform to varieties of Fano type.
In the following example, we give a negative answer to this question even if we assume the fibers are $\Q$-Gorenstein.

\begin{eg}
Let $m,n, I, R, t $ and $T$ be as in Example \ref{counter eg local}.
First we observe that $R$ is an $\N$-graded $\Z$-algebra with respect to the grading 
\[
\deg A=nm, \; \deg B=2m, \; \deg C=2mn, \; \deg D=2mn,\; \deg E =2n.
\]

We will show that the generic fiber $\mathcal{Y}_{\eta}$ of a flat projective morphism $\mathcal{Y}:=\Proj R \to T$ is not of Fano type, while $\mathcal{Y}_{\eta}$ is $\Q$-Gorenstein and the fiber $\mathcal{Y}_t$ over the closed point $t \in T$ is $\Q$-Gorenstein globally $F$-regular.

Since the graded ring $R/(3)$ is strongly $F$-regular (see Example \ref{counter eg local}), every Veronese subring of $R/(3)$ is also strongly $F$-regular by \cite[Theorem 3.1 (e)]{HH}. 
Noting that $\mathcal{Y}_t = \Proj R/(3)$, one can pick an ample invertible sheaf $L$ on $\mathcal{Y}_t$ such that the section ring $\bigoplus_{\ell \ge 0} H^0(\mathcal{Y}_t, L^{\otimes \ell})$ is isomorphic to a Veronese subring of $R/(3)$.
It, therefore, follows from \cite[Proposition 5.3]{SS} that $\mathcal{Y}_t$ is globally $F$-regular.
Moreover, $\mathcal{Y}_t$ is $\Q$-Gorenstein by Remark \ref{remark rat} (ii) and Remark \ref{SFR to F-rat}, because $\mathcal{Y}_t$ has only strongly $F$-regular singularities and $\dim \mathcal{Y}_t=2$.
As we have seen in the proof of Corollary \ref{mixed klt local}, the generic fiber $\mathcal{Y}_{\eta}$ is also $\Q$-Gorenstein.

We finally show that $\mathcal{Y}_{\eta}$ is not of Fano type.
Assume to the contrary that $\mathcal{Y}_{\eta}$ is of Fano type, and then so is $\mathcal{Y}_{\eta} \times_{\Spec \Q} \Spec \C$. 
Since $\mathcal{Y}_{\eta} \times_{\Spec \Q} \Spec \C=\Proj (R \otimes_\Z \C)$, by Proposition \ref{Fano klt}, there exists an integer $u \ge 1$ such that the spectrum of the $u$-th Veronese subring of $R \otimes_\Z \C$ is klt in the sense of de Fernex-Hacon.
It follows from an argument used in Example \ref{counter eg local} that the factor ring $R^{(u)}/(p)$ of the $u$-th Veronese subring $R^{(u)}$ of $R$ modulo a general prime $p$ is strongly $F$-regular. 
We may assume that $p$ is large enough so that $p>3$ and $u$ is not divisible by $p$.
Then the extension $R^{(u)}/(p) \subseteq R/(p)$ is \'etale in codimension one, and therefore $R/(p)$ is strongly $F$-regular by \cite[Theorem 2.7]{Wat}, which contradicts \cite[Theorem 1.1]{Sin}. 
Thus, $\mathcal{Y}_{\eta}$ is not of Fano type.
\end{eg}

\appendix
\section{Relative test ideals for triples}\label{relative test ideals}

In this appendix, we generalize the theory of relative test ideals, introduced in \cite{PSZ} for pairs, to the case of triples. 
In order to obtain a stabilization result (Proposition \ref{stabilization2}) in this case, which is a bit more subtle than the case of pairs (\cite[Lemma 4.2]{PSZ}), we use an argument similar to the proof of \cite[Proposition 3.8]{Sat2}. 

Let $R$ be an integral domain of characteristic $p>0$ and $q=p^e$ be a power of $p$.
We fix an algebraic closure $\overline{\mathrm{Frac}(R)}$ of the fractional field $\mathrm{Frac}(R)$ of $R$, and let $R^{1/q}$  
be the ring of $q$-th roots of elements in $R$, that is, 
\[
R^{1/q}  = \{ x \in \overline{\mathrm{Frac}(R)} \mid x^q \in R \}.
\]
Given an $R$-module $M$, the ring isomorphism $R^{1/q} \xrightarrow{\sim} R$ sending $x$ to $x^q$ induces an $R^{1/q}$-module structure on $M$, which is denoted by $M^{1/q}$. 
$R^{1/q}$-modules are considered as $R$-modules via the natural inclusion $R \hookrightarrow R^{1/q}$. 
When $M^{1/q}$ is regarded as an $R$-module in this way, it is nothing but the push-forward $F^e_* M$ of $M$ by the $e$-th iterated Frobenius morphism $F^e: \Spec R \to \Spec R$. 

\begin{lem}\label{projection formula}
Let $M$ be a module over an integral domain $R$ of characteristic $p>0$. 
\begin{enumerate}[label=$(\arabic*)$]
\item For an invertible $R$-module $L$, we have an isomorphism 
\[
 L \otimes_R M^{1/q} \cong (L^{\otimes q} \otimes_R M)^{1/q}
\] of $R^{1/q}$-modules.
\item For an ideal $\ba \subseteq R$, we have the equality 
\[
\ba \cdot M^{1/q} = (\ba^{[q]} M )^{1/q},
\]
where $\ba^{[q]} \subseteq R$ is the ideal generated by the $q$-th powers of all elements of $\ba$. 
\end{enumerate}
\end{lem}

\begin{proof}
Since $(F^e)^* L \cong L^{\otimes q}$ as $R$-modules, it follows from the projection formula that there exists an $R$-module isomorphism
\[
f: L \otimes_R M^{1/q} \xrightarrow{\sim} ( (F^e)^* L \otimes_R M)^{1/q}  \xrightarrow{\sim} (L^{\otimes q} \otimes_R M)^{1/q}.
\]
It is straightforward to verify that $f$ is an $R^{1/q}$-module homomorphism, which shows (1). 
The assertion (2) is obvious.
\end{proof}

Let $N$ be an invertible $R$-module and $\gamma: N^{1/q} \to R$ be an $R$-module homomorphism.
For each integer $n \ge 0$, we set $N^{(n)} : = N^{\otimes \frac{q^n-1}{q-1}}$ and $\gamma^n : (N^{(n)})^{1/q^n} \to R$ denotes the $R$-module homomorphism defined inductively by the following composite: 
\begin{eqnarray*}
\gamma^n: (N^{(n)})^{1/q^n} 
=(((N^{(n-1)})^{\otimes q} \otimes_R N)^{1/q})^{1/q^{n-1}} 
& \xrightarrow{\hspace*{1.85em} \sim \hspace*{1.85em}} & (N^{(n-1)} \otimes_R N^{1/q})^{1/q^{n-1}} \\
& \xrightarrow{(\mathrm{id} \otimes \gamma)^{1/q^{n-1}}} & (N^{(n-1)})^{1/q^{n-1}}\\ 
& \xrightarrow{\hspace*{1.35em} \gamma^{n-1} \hspace*{1.35em}} & R. 
\end{eqnarray*}
From now on, by abuse of notation, the map $(N^{(n)})^{1/q^n} \to (N^{(n-1)})^{1/q^{n-1}}$ is also denoted by $\gamma$. 

\begin{defn}\label{defn absolute tau}
Suppose that $R$ is an integral domain of characteristic $p>0$ and $\gamma : N^{1/q} \to R$ is an $R$-module homomorphism, where $N$ is an invertible $R$-module and $q=p^e$ is a power of $p$. 
Let $I$ and $\ba$ be nonzero ideals of $R$ and $\lambda > 0$ be a real number.
The \textit{test ideal} $\tau(X, \gamma I , \ba^{\lambda})$ of $(X:= \Spec R, \gamma, \ba^{\lambda})$ with respect to $I$ is then defined as 
\[
\tau(X, \gamma I , \ba^{\lambda}) =\sum_{i \ge 0} \gamma^i( ( \ba^{\lceil q^i \lambda \rceil }I N^{(i)})^{1/q^i}) \subseteq R.
\]
\end{defn}

\begin{eg}\label{phi and divisor}
With notation as above, we suppose in addition that $R$ is an $F$-finite normal domain with dualizing complex $\omega^{\bullet}_R$ such that $F^! \omega_R^{\bullet} \cong \omega_R^{\bullet}$ and there exists an effective $\Q$-Weil divisor $\Delta$ on $X:=\Spec R$ such that $(q-1)(K_X+\Delta)$ is Cartier and $I \subseteq \tau(X,\Delta)$.
If $N = \sO_X((1-q)(K_X+\Delta))$ and $\gamma$ is obtained from the equivalence relation $(\star)$ in \cite[Paragraphs after Definition 2.4]{BSTZ} (see also \cite[Theorem 3.11]{Sch2}), then 
\[
\tau(X, \gamma I, \ba^{\lambda}) = \tau(X, \Delta, \ba^{\lambda}).
\]
Indeed, after localization, we may assume that $R$ is local and $N =R$.
Since $\gamma^n$ is a generator of the free $F^{en}R$-module 
$\Hom_R(F^{en}_*\sO_{X}((q^n-1)\Delta), R)$ 
of rank one, the assertion follows from Remark \ref{test element remark} and \cite[Proposition 4.6]{Sch3}.
\end{eg}

Let $f: A \to R$ be a ring homomorphism of integral domains of characteristic $p>0$.
We write $R_{A^{1/q}} : = R \otimes_A A^{1/q}$, and note that the inclusion $R \to R^{1/q}$ and the natural morphism $f^{1/q} : A^{1/q} \to R^{1/q}$ induces the natural morphism $R_{A^{1/q}} \to R^{1/q}$.
\[
\xymatrix{
R  \ar[r] & R_{A^{1/q}} \ar[r] & R^{1/q} \\
A  \ar[r] \ar^-{f}[u] & A^{1/q} \ar[u] \ar_-{f^{1/q}}[ru] & 
}\]
Given an $R^{1/q}$-module $M$, we view $M$ as an $R_{A^{1/q}}$-module via this ring homomorphism. 

\begin{setting}\label{relative test setting}
Suppose that $A$ is a Noetherian integral domain of characteristic $p>0$, $R$ is an integral domain flat and essentially of finite type over $A$, $\lambda > 0$ is a real number, and $I, \ba \subseteq R$ are nonzero ideals.
Let $\phi : L^{1/q} \to R_{A^{1/q}}$ be an $R_{A^{1/q}}$-module homomorphism, where $L$ is an invertible $R$-module and $q=p^e$ is a power of $p$. 
For each integer $n \ge 0$, we set $B_n : = R_{A^{1/q^n}}$, $L^{(n)} : = L^{\otimes \frac{q^n-1}{q-1}}$, $X: = \Spec R$, and $V : = \Spec A$. 
\end{setting}

For each integer $n>0$, we define a $B_n$-homomorphism $\tilde{\phi}_n : (L \otimes_R B_{n-1})^{1/q} \to B_n$ by
\[
\tilde{\phi}_n : (L \otimes_R B_{n-1})^{1/q} \cong L^{1/q} \otimes_{A^{1/q}} A^{1/q^{n}} \xrightarrow{\phi \otimes \mathrm{id}} B_1 \otimes_{A^{1/q}} A^{1/q^{n}} \cong B_n.
\]
Moreover, we define a $B_n$-homomorphism $\phi^n : (L^{(n)})^{1/q^n} \to B_n$ inductively as follows:
\[
\phi^n : (L^{(n)})^{1/q^n} \cong  (L \otimes_R (L^{(n-1)})^{1/q^{n-1}})^{1/q} \xrightarrow{(\mathrm{id} \otimes \phi^{n-1})^{1/q}} (L \otimes_R B_{n-1})^{1/q} \xrightarrow{\tilde{\phi}_n} B_n.
\]
For integers $n \ge i \ge 0$, let $a_{i,n} : B_i \to B_n$  
be the ring homomorphism induced by the natural inclusion $A^{1/q^i} \hookrightarrow A^{1/q^n}$.

\begin{defn}[$\textup{cf. \cite[Definition 4.3]{PSZ}}$]
With notation as in Setting \ref{relative test setting}, the \emph{$n$-th limiting relative test ideal} $\tau_n(X/V, \phi I, \ba^{\lambda})$ of $(X/V, \phi, \ba^{\lambda})$ with respect to $I$ is defined as
\[
\tau_n(X/V, \phi I, \ba^{\lambda}) : = \sum_{i=0}^n \phi^i( (\ba^{\lceil q^i \lambda \rceil} I L^{(i)})^{1/q^i}) B_n \subseteq B_n.
\]
\end{defn}

\begin{lem}\label{phi(tau)}
With notation as in Setting \ref{relative test setting}, we have 
\[
\tilde{\phi}_{n+1}((L \otimes_R \tau_{n}(X/V, \phi I, \ba^{\lambda}))^{1/q}) + (\ba^{\lceil \lambda/q \rceil} I) B_{n+1} = \tau_{n+1}(X/V, \phi I, \ba^{\lambda/q}).
\]
\end{lem}

\begin{proof}
Note that we have the following commutative diagram:
\[
\xymatrix{
(L \otimes_R B_i)^{1/q}  \ar^-{\tilde{\phi}_{i+1}}[r]\ar_-{(\mathrm{id} \otimes a_{i,n})^{1/q}}[d] & B_{i+1} \ar^{a_{i+1, n+1}}[d]\\
(L \otimes_R B_n)^{1/q} \ar^-{\tilde{\phi}_{n+1}}[r] &  B_{n+1}.
}\]
Therefore, 
\begin{align*}
\tilde{\phi}_{n+1}( ( L \otimes_R \phi^i((\ba^{\lceil q^i \lambda \rceil} I L^{(i)})^{1/q^i}) B_{n})^{1/q}) 
&= \tilde{\phi}_{i+1}((L \otimes_R \phi^i((\ba^{\lceil q^i \lambda \rceil} I L^{(i)})^{1/q^i}))^{1/q}) B_{n+1} \\
&= \phi^{i+1}( (\ba^{\lceil q^i \lambda \rceil} I L^{(i+1)})^{1/q^{i+1}}) B_{n+1} \\
&= \phi^{i+1}( (\ba^{\lceil q^{i+1} (\lambda/q) \rceil} I L^{(i+1)})^{1/q^{i+1}}) B_{n+1},
\end{align*}
which implies the assertion.
\end{proof}

The notation $\mu(\ba)$ denotes the minimal number of generators for the ideal $\ba$.

\begin{lem}\label{Skoda}
With notation as in Setting \ref{relative test setting},  assume that $\lambda > \mu(\ba)-1$. 
Then 
\[
\tau_n(X/V, \phi I , \ba^{\lambda})\ba =\tau_n(X/V, \phi I , \ba^{\lambda+1}).
\] 
\end{lem}

\begin{proof}
Since $a_{i, n}: B_i \to B_n$ and $\phi^i$ are $R$-homomorphisms, 
\begin{align*}
\tau_n(X/V, \phi I , \ba^{\lambda}) \ba 
&= \sum_{i=0}^n \phi^i ( \ba (\ba^{\lceil q^i \lambda \rceil} I L^{(i)})^{1/q^i}) B_n \\
&= \sum_{i=0}^n \phi^i ( ( \ba^{[q^i]} \ba^{\lceil q^i \lambda \rceil} I L^{(i)})^{1/q^i}) B_n \\
&= \sum_{i=0}^n \phi^i ( ( \ba^{q^i + \lceil q^i \lambda \rceil} I L^{(i)})^{1/q^i}) B_n\\
&=\tau_n(X/V, \phi I , \ba^{\lambda+1}), 
\end{align*}
where the third equality follows from \cite[Lemma 2.11]{Sat2} because $\lambda > \mu(\ba)-1$.
\end{proof}

\begin{prop}\label{stabilization1}
With notation as in Setting \ref{relative test setting}, assume that $\lambda> \mu(\ba) -1$ and $(q-1)\lambda$ is an integer.
If the equality  
\[
\tau_{n-1}(X/V, \phi I, \ba^{\lambda}) B_n = \tau_n(X/V, \phi I, \ba^{\lambda})
\] 
holds for some integer $n \ge 1$, then
\[
\tau_{m-1}(X/V, \phi I, \ba^{\lambda}) B_m = \tau_m(X/V, \phi I, \ba^{\lambda})
\]
holds for every integer $m>n$. 
\end{prop}

\begin{proof}
It is enough to check the equality when $m=n+1$.
Since $\tau_{n-1}(X/V, \phi I , \ba^{\lambda}) B_n = \tau_n(X/V, \phi I , \ba^{\lambda})$, the commutative diagram in the proof of Lemma \ref{phi(tau)} yields 
\[
\widetilde{\phi}_n( (L \otimes_R \tau_{n-1}(X/V, \phi I , \ba^{\lambda})\ba^{(q-1)\lambda})^{1/q}) B_{n+1} = \widetilde{\phi}_{n+1}( (L \otimes_R (\tau_n(X/V, \phi I , \ba^{\lambda})\ba^{(q-1)\lambda})^{1/q}).
\]
Combining this with Lemmas \ref{phi(tau)} and \ref{Skoda}, we complete the proof.
\end{proof}

With notation as in Setting \ref{relative test setting}, suppose that $V' : = \Spec A' \to V$ is a morphism from a Noetherian integral affine scheme such that $X' : = X \times_V V'$ is integral.
We set $R ' : = R \otimes_A A'$, $L' : = L \otimes_R R'$, $I' : = I R'$, $\ba' : = \ba R'$, and  $B'_n : = R' \otimes_{A'} {A'}^{1/q^n}$.
Let $\phi': (L')^{1/q} \to B'_1$ be the morphism induced by $\phi$ $($see \cite[Subsection 2.16]{PSZ} for details$)$.

\begin{lem}[$\textup{cf. \cite[Theorem 4.5]{PSZ}}$]\label{relative tau base change}
With notation as above, we have
\[
\tau_n(X'/V', \phi' I', {\ba'}^{\lambda}) = \tau_n(X/V, \phi I, \ba^{\lambda}) B'_n.
\]
\end{lem}

\begin{proof}
The assertion follows from an argument similar to the proof of \cite[Theorem 4.5]{PSZ}.
\end{proof}

Proposition \ref{stabilization1} with the aid of Lemma \ref{relative tau base change} yields the following stabilization result. 
This strategy of the proof goes back to \cite[Proposition 3.8]{Sat2}. 
\begin{prop}[$\textup{cf. \cite[Lemma 4.2]{PSZ}}$]\label{stabilization2}
With notation as in Setting \ref{relative test setting},  assume that $\lambda> \mu(\ba) -1$ and $(q-1)\lambda$ is an integer.
There exist a dense open subset $U \subseteq V$ and an integer $n_0$ such that for any integer $n \ge n_0$ and any morphism $V' =\Spec A' \to V$ as in Lemma \ref{relative tau base change} whose image is contained in $U$, we have
\[
\tau_{n-1}(X'/V', \phi' I', {\ba'}^{\lambda} ) B'_n  = \tau_n(X'/V', \phi' I', {\ba'}^{\lambda}).
\]
\end{prop}

\begin{proof}
By an argument similar to the proof of \cite[Proposition 3.3]{PSZ}, there exist a dense open subset $U \subseteq V$ and an integer $n_0$ such that 
\[
(\tau_{n_0-1}(X/V, \phi I, \ba^{\lambda} ) B_n)|_U = \tau_{n_0}(X/V, \phi I, \ba^{\lambda} )|_U.
\]
The assertion then follows from Lemma \ref{relative tau base change} and Proposition \ref{stabilization1}, 
\end{proof}

With notation as in Setting \ref{relative test setting}, assume that the morphism $X \to V=\Spec A$ is normal and $A$ is an $F$-finite regular integral domain with dualizing complex $\omega_A^{\bullet}$ such that $F^{!} \omega_A^{\bullet} \cong \omega_A^{\bullet}$.
We note that in this case, $B_m=R_{A^{1/q^m}}=R \otimes_A A^{1/q^m}$ is normal for every $m$ by \cite[p.184 Corollary]{Mat}.

We fix a canonical divisor $K_V$ of $V=\Spec A$ and write $M : = A((1-q) K_V)$.
Let $\Psi : M^{1/q} \to A$ be a generator of the free $F^e_* A$-module $\Hom_A(M^{1/q}, A)$ of rank one. 
For each integer $m \ge 1$, we define the $B_m$-module 
$N_m : = L \otimes_A M^{1/q^m}$ 
and the $B_m$-module homomorphism $\gamma_m: N_m^{1/q} \to B_m$ as the composite map 
\[
\gamma_m: N_m^{1/q} \cong L^{1/q} \otimes_{A^{1/q}} M^{1/q^{m+1}} \xrightarrow{\phi \otimes \Psi^{1/q^m}} B_1 \otimes_{A^{1/q}} A^{1/q^m} \cong B_m.
\]

Given a scheme $T$ of characteristic $p>0$ and an integer $n \ge 0$, the $n$-th iterated Frobenius morphism $F: T \to T$ induces a $T$-scheme structure on $T$, which is denoted by  $T^n$.
In our setting, $V^{em} \cong \Spec A^{1/q^m}$ as $V$-schemes and $X_{V^{em}} : = X \times_V V^{em} \cong \Spec B_m$. 

\begin{lem}\label{phi(tau)2}
With notation as above, we set $\mathfrak{c}_m : = \tau (X_{V^{em}}, \gamma_{m} (IB_{m}) , (\ba B_{m})^{\lambda}) \subseteq B_m$. Then 
\[
\tilde{\phi}_{m+1}((L \otimes_R \mathfrak{c}_m)^{1/q}) + (\ba^{\lceil \lambda/q \rceil} I) B_{m+1} = \tau(X_{V^{e(m+1)}}, \gamma_{m+1} (I B_{m+1}), (\ba B_{m+1})^{\lambda/q}).
\]
\end{lem}

\begin{proof}
First note that for any integer $n \ge 0$,  
\begin{align*}
N_{m+1}^{(n+1)} &\cong L^{\otimes q^n} \otimes_R L^{(n)} \otimes_A (M^{1/q^m})^{(n)} \otimes_{A^{1/q^m}} M^{1/q^{m+1}}\\
&\cong  L^{\otimes q^n} \otimes_R N_m^{(n)} \otimes_{A^{1/q^m}} M^{1/q^{m+1}}, 
\end{align*}
where $(M^{1/q^m})^{(n)}:=(M^{1/q^m})^{\otimes_{A^{1/q^m}}\frac{q^n-1}{q-1}}$. 
We define the $B_m^{1/q^{n+1}}$-module homomorphism $f_n : (N_{m+1}^{(n+1)})^{1/q^{n+1}} \to (L \otimes_R (N_m^{(n)})^{1/q^n})^{1/q}$ as the composite map 
\[
f_n : (N_{m+1}^{(n+1)})^{1/q^{n+1}} \xrightarrow{(\mathrm{id} \otimes \mathrm{id} \otimes \Psi^{1/q^{m}})^{1/q^{n+1}}}(L^{\otimes q^n} \otimes_R N_m^{(n)})^{1/q^{n+1}}
\xrightarrow{\sim} (L \otimes_R (N_m^{(n)})^{1/q^n})^{1/q} .
\]
Then the following diagram commutes.
\[
\xymatrix{
(N_{m+1}^{(n+1)})^{1/q^{n+1}}  \ar^-{\gamma_{m+1}}[d]\ar_-{f_n}[r]
& (L \otimes_R (N_m^{(n)})^{1/q^n})^{1/q}  \ar^-{(\mathrm{id} \otimes \gamma_m)^{1/q}}[d] \\
(N_{m+1}^{(n)})^{1/q^{n}}  \ar^-{\gamma_{m+1}}[d]\ar_-{f_{n-1}}[r]
& (L \otimes_R (N_m^{(n-1)})^{1/q^{n-1}})^{1/q}  \ar^-{(\mathrm{id} \otimes \gamma_m)^{1/q}}[d] \\
\vdots \ar[d] & \vdots \ar[d]\\
(N_{m+1})^{1/q}  \ar^-{\gamma_{m+1}}[d]\ar_-{f_0}[r] &
(L \otimes_R B_m)^{1/q} \ar^-{\tilde{\phi}_{m+1}}[ld]\\
B_{m+1} &
}\]

Since $A$ is regular, the morphism $\Psi$ is surjective and so is $f_n$.
Therefore, 
\[
\tilde{\phi}_{m+1}((L \otimes_R \gamma_m^n( (\ba^{\lceil q^n \lambda \rceil} I N_m^{(n)} )^{1/q^n} ))^{1/q}) = \gamma_{m+1}^{n+1}( (\ba^{\lceil q^n \lambda \rceil} I N_{m+1}^{(n+1)} )^{1/q^{n+1}} ),
\]
which proves the assertion.
\end{proof}

\begin{prop}[$\textup{cf. \cite[Theorem 4.13]{PSZ}}$]\label{tau vs fiber}
With notation as above, assume further that there exists an integer $l>0$ such that $q^l(q-1) \lambda$ is an integer.
Then there exist a dense open subset $U \subseteq V$ and an integer $n_1 \ge 1$ satisfying the following: for any integer $m \ge n_1$ and any morphism $V' \to V$ with image in $U$, where $V'=\Spec A'$ is an $F$-finite regular integral affine scheme with dualizing complex $\omega_{V'}^{\bullet}$ such that $F^!\omega_{V'}^{\bullet} \cong \omega_{V'}^{\bullet}$ and $X':=X \times_V V'$ is an integral scheme, we have
\[
\tau_{m}(X'/V',  \phi' I', {\ba'}^{\lambda} ) = \tau(X'_{V'^{em}}, \gamma'_m (I' B'_m), (\ba' B'_m)^{\lambda} ).
\]
Here $\gamma'_m:(N'_m)^{1/q}:=(L' \otimes_{A'} A'((1-q)K_{V'})^{1/q^m})^{1/q} \to B'_m$ denotes the $B'_m$-module homomorphism induced by $\phi':(L')^{1/q} \to B_1'$.
\end{prop}

\begin{proof}
By Lemmas \ref{phi(tau)} and \ref{phi(tau)2}, the problem is reduced to the case where $(q-1) \lambda$ is an integer and $\lambda> \mu(\ba) -1$. 
In this case, taking $U$ and $n_0$ as in Proposition \ref{stabilization2}, we see that $n_1 : =n_0$ satisfies the assertion by an argument similar to the proof of Proposition \ref{stabilization1}. 
The reader is referred to the proof of \cite[Theorem 4.13]{PSZ} for more details.
\end{proof}

\begin{thm}[$\textup{cf. \cite[Corollary 4.15]{PSZ}}$]\label{absolute tau vs fiber}
Suppose that $V$ is an $F$-finite regular integral scheme of characteristic $p>0$, $f: X \to V$ is a normal morphism essentially of finite type from a normal integral scheme $X$, and $\Delta$ is an effective $\Q$-Weil divisor on $X$ such that $K_X+\Delta$ is $\Q$-Cartier with index not divisible by $p$.  
Let $\ba \subseteq \sO_X$ be a nonzero coherent ideal sheaf and $\lambda>0$ be a rational number.
Then there exist a dense open subset $U \subseteq V$ and an integer $n_2 \ge 1$ such that 
for every positive multiple $n$ of $n_2$  
and every perfect point $u : \Spec k \to U$ of $U$, one has 
\[
\tau(X_{V^{n}}, h^* \Delta, (\ba \sO_{X_{V^{n}}}) ^\lambda) \sO_{X_u} = \tau(X_u, \Delta|_{X_u}, (\ba \sO_{X_u})^\lambda),
\]
where $h^* \Delta$ is the pullback of $\Delta$ by the projection $h: X_{V^{n}} : = X \times_V V^{n} \to X$.
\end{thm}

\begin{proof}
Shrinking $V$ if necessary, we may assume that $\Delta$ does not contain any fiber of $f$ in its support and $V$ is affine with dualizing complex $\omega_V^{\bullet}$ such that $F^{!} \omega_V^{\bullet} \cong \omega_V^{\bullet}$.
Taking an affine open covering of $X$, we may assume that $X$ is also affine.
Fix a relative canonical divisor $K_{X/V}$ on $X$.
Since $V$ is Gorenstein, it follows from \cite[Lemma 0BZL]{Sta} that $K_{X/V}$ is a canonical divisor of $X$, and in particular, by \cite[V. Theorem 3.1]{Hart}, the $\Q$-Weil divisor $K_{X/V} + \Delta$ is $\Q$-Cartier with index not divisible by $p$. 
Choose a power $q$ of $p$ such that $(q-1)(K_{X/V} + \Delta)$ is Cartier and $q^l(q-1) \lambda$ is an integer for some $l \ge 1$.
The assertion then follows from Proposition \ref{tau vs fiber}, Lemma \ref{relative tau base change}, and Example \ref{phi and divisor}.
The reader should compare this proof with that of \cite[Corollary 4.15]{PSZ}. 
\end{proof}


\end{document}